\documentclass[preprint]{elsarticle}

\usepackage{comment}
\usepackage{amssymb}
\usepackage{amsmath}
\usepackage{url}
\usepackage{hyperref}
\usepackage{lineno}

\newtheorem{theorem}{Theorem}
\newtheorem{lemma}[theorem]{Lemma}
\newtheorem{proposition}[theorem]{Proposition}
\newtheorem{corollary}[theorem]{Corollary}
\newdefinition{definition}[theorem]{Definition}
\newproof{proof}{Proof}

\begin{document}

\begin{frontmatter}

\title{List homomorphism problems for signed trees\tnoteref{funding}
}

\tnotetext[funding]{The first author received funding from the European Union’s Horizon 2020 project H2020-MSCA-RISE-2018: Research and Innovation Staff Exchange and from the Charles University Grant Agency project 1580119. The second author was supported by his NSERC Canada Discovery Grant. The fourth and fifth author were also partially supported by the fourth author’s NSERC Canada Discovery Grant. The fifth author was also supported by the Charles University Grant Agency project 1198419 and by the Czech Science Foundation (GA-ČR) project 19-17314J.}

\author[1]{Jan Bok\corref{cor1}}
\cortext[cor1]{Corresponding author}
\ead{bok@iuuk.mff.cuni.cz}

\author[2]{Richard Brewster}
\ead{rbrewster@tru.ca}

\author[3]{Tom\' as Feder}
\ead{tomas@theory.stanford.edu}

\author[4]{Pavol Hell}
\ead{pavol@cs.sfu.ca}

\author[5]{Nikola Jedli\v ckov\' a}
\ead{jedlickova@kam.mff.cuni.cz}

\address[1]{Computer Science Institute, Faculty of Mathematics and Physics, Charles University, Czech Republic}
\address[2]{Department of Mathematics and Statistics, Thompson Rivers University, Canada}
\address[3]{268 Waverley St., Palo Alto, USA}
\address[4]{School of Computing Science, Simon Fraser University, Canada}
\address[5]{Department of Applied Mathematics, Faculty of Mathematics and Physics, Charles University, Czech Republic}


\begin{abstract}
We consider homomorphisms of signed graphs from a computational perspective.
In particular, we study the list homomorphism problem seeking a homomorphism
of an input signed graph $(G,\sigma)$, equipped with lists $L(v) \subseteq
V(H), v \in V(G)$, of allowed images, to a fixed target signed graph
$(H,\pi)$. The complexity of the similar homomorphism problem without lists
(corresponding to all lists being $L(v)=V(H)$) has been previously classified
by Brewster and Siggers, but the list version remains open and appears
difficult. We illustrate this difficulty by classifying the complexity of the
problem when $H$ is a tree (with possible loops). The tools we develop will 
be useful for classifications of other classes of signed graphs, and in a future 
companion paper we will illustrate this by using them to classify the complexity 
for certain irreflexive signed graphs. The structure of the signed trees in the 
polynomial cases is interesting, suggesting that the class of general signed 
graphs for which the problems are polynomial may have nice structure, analogous 
to the so-called bi-arc graphs (which characterized the polynomial cases of list 
homomorphisms to unsigned graphs).
\end{abstract}

\end{frontmatter}

\section{Motivation}\label{sone}

We investigate a problem at the confluence of two popular topics -- graph
homomorphisms and signed graphs. Their interplay was first considered in an
unpublished manuscript of Guenin~\cite{guenin}, and has since become an
established field of study~\cite{rezazasla}. 

We now introduce the two topics separately. In the study of computational
aspects of graph homomorphisms, the central problem is one of existence --
does an input graph $G$ admit a homomorphism to a fixed target graph $H$? 
(The graphs considered here are undirected graphs with possible loops but no
parallel edges.) 
This is known as the {\em graph homomorphism problem}. It was shown in~\cite{hn}
that this problem is polynomial-time solvable when $H$ has a loop or is
bipartite, and is NP-complete otherwise. This is known as the {\em dichotomy}
of graph homomorphisms (see~\cite{hn1}). The {\em core} of a
graph $H$ is a subgraph of $H$ with the smallest number of vertices to which
$H$ admits a homomorphism; note that such a subgraph is unique up to
isomorphism. A graph with a loop has a vertex with a loop as its core, and 
a (non-empty) bipartite graph has an edge as its core. Thus an equivalent way 
of stating the graph dichotomy result is that the problem is polynomial-time
solvable when the core of $H$ has at most one edge, and is NP-complete
otherwise.

Now suppose the input graph $G$ is
equipped with lists, $L(v) \subseteq V(H), v \in V(G),$ and we ask if there is
a homomorphism $f$ of $G$ to $H$ such that each $f(v) \in L(v)$. This is known
as the {\em graph list homomorphism problem}. This problem also has a dichotomy
of possible complexities~\cite{feder2003bi} -- it is polynomial-time solvable
when $H$ is a so-called bi-arc graph and is NP-complete otherwise. Bi-arc
graphs have turned out to be an interesting class of graphs; for instance,
when $H$ is a reflexive graph (each vertex has a loop), $H$ is a bi-arc graph
if and only if it is an interval graph~\cite{feder1998list}. 

These kinds of complexity questions found their most general formulation in
the context of constraint satisfaction problems. The Feder-Vardi dichotomy 
conjecture~\cite{fv} claimed that every constraint satisfaction problem with 
a fixed template $H$ is polynomial-time solvable or NP-complete. After a 
quarter century of concerted effort by researchers in theoretical computer 
science, universal algebra, logic, and graph theory, the conjecture was 
proved in 2017, independently by Bulatov~\cite{bula} and Zhuk~\cite{zhuk}. 
This exciting development focused research attention on additional homomorphism 
type dichotomies, including ones for signed graphs~\cite{BFHN,dichotomy,FN14}.

The study of signed graphs goes back to~\cite{harary,hararykabell}, and has
been most notably investigated in~\cite{zav81,zav82b,zav82a,Z97,zavsurvey},
from the point of view of colourings, matroids, or embeddings. Following
Guenin, homomorphisms of signed graphs have been pioneered in~\cite{brewgrav}
and~\cite{nasrolsop}. 
The computational aspects of existence of homomorphisms
in signed graphs --- given a fixed signed graph $(H,\pi)$, does an input signed
graph $(G,\sigma)$ admit a homomorphism to $(H,\pi)$ --- were studied
in~\cite{BFHN,FN14}, and eventually a complete dichotomy classification was
obtained in~\cite{dichotomy}.  It is surprisingly similar to the second way we
stated the graph dichotomy result above, see Theorem~\ref{thm:dichotomy},
and the discussion following it.

Although typically homomorphism problems tend to
be easier to classify with lists than without lists (lists allow for recursion
to subgraphs), the complexity of the list homomorphism problem for signed
graphs appears difficult to classify~\cite{bordeaux,dichotomy}. If the analogy
to (unsigned) graphs holds again, then the tractable cases of the problem
should identify an interesting class of signed graphs, generalizing bi-arc
graphs. In this paper, we begin the exploration of this concept, focusing on
the case of signed trees. We find that there is interesting structure to the 
tractable cases.

\section{Terminology and notation}

A {\em signed graph} is a graph $G$, with possible loops and multiple edges
(at most two loops per vertex and at most two edges between a pair of vertices),
together with a mapping $\sigma\colon E(G) \to \{+, -\}$, assigning a sign ($+$ or
$-$) to each edge and each loop of $G$, so that different loops at a vertex have 
different signs, and similarly for different edges between the same two vertices. 
For convenience, we shall usually consider an edge to mean an edge or a loop,
and to emphasize otherwise we shall call it a {\em non-loop edge}. Thus we can
say, for example, that each edge of a signed graph has a sign, meaning both
loops and non-loop edges. We denote a signed graph by $(G,\sigma)$, and
call $G$ its {\em underlying graph} and $\sigma$ its {\em signature}. When the
signature name is not needed, we denote the signed graph $(G,\sigma)$ by
$\widehat{G}$ to emphasize that it has a signature even though we do not give
it a name. We will usually view signs of edges as colours, and call positive 
edges {\em blue}, and negative edges {\em red}. It will be convenient to call a 
red-blue pair of edges with the same endpoint(s) a {\em bicoloured edge}
(this includes loops as well as non-loop edges); however, formally they are 
two distinct edges. By contrast, we call edges that are not part of such a pair 
{\em unicoloured}; moreover, when we refer to an edge as blue or red we shall 
always mean the edge is unicoloured blue or red. We also call an edge {\em 
at least blue} if it is either blue or bicoloured, and similarly for {\em at least red} 
edges. The terms {\em at least positive} and {\em at least negative} are used 
in the same sense. Treating a pair of red-blue edges as one bicoloured 
edge is advantageous in many descriptions, but introduces an ambiguity 
when discussing walks, since a walk in a signed graph could be seen 
as a sequence of incident vertices and edges, and so selecting just 
one edge from a red-blue pair, or it could be interpreted as a
sequence of consecutively adjacent vertices, and hence contain some 
bicoloured edges. This creates particular problem for cycles, since in the 
former view, a bicoloured edge would be seen as a cycle of length two,
with one red edge and one blue edge. In the literature, the former approach 
is more common, but here we take the latter approach. Of course, the two
views coincide if only walks of unicoloured edges are considered. The sign 
of a walk consisting of unicoloured edges $\widehat{G}$ is the product of 
the signs of its edges.  Thus a walk of unicoloured edges is {\em negative} 
if it has an odd number of negative (red) edges, and {\em positive} if it has an 
even number of negative (red) edges. In the case of unicoloured cycles, we also 
call a negative cycle {\em unbalanced} and a positive cycle {\em balanced}. Note
that a vertex with a red loop is a cycle with one negative edge, and hence is unbalanced.
A {\em uni-balanced signed graph} is a signed graph without unbalanced cycles, i.e.,
a signed graph in which all unicoloured cycles (if any) have an even number of 
red edges. An {\em anti-uni-balanced signed graph} is a signed graph in which each 
unicoloured cycle has an even number of blue edges. Thus we have a symmetry 
to viewing the signs as colours, in particular $\widehat{G}$ is uni-balanced if and only 
if $\widehat{G}'$, obtained from $\widehat{G}$ by exchanging the colour of each 
edge, is anti-uni-balanced. We introduce the qualifier "uni-" because the notion
of a {\em balanced} signed graph is well established in the literature: it means a
signed graph without any unbalanced cycles in the classical view, including the
two-cycles formed by red-blue pairs of edges. Thus a balanced signed graph is
a uni-balanced signed graph without bicoloured edges and loops.

We now define the \emph{switching} operation. This operation can be applied
to any vertex of a signed graph and it negates the signs of all its incident
non-loop edges. (The signs of loops are unchanged by switching.) We say that
two signatures $\sigma_1, \sigma_2$ of a graph $G$ are \emph{switching
equivalent} if we can obtain $(G,\sigma_2)$ from $(G,\sigma_1)$ by a sequence
of switchings. In that case we also say that the two signed graphs $(G,\sigma_1)$ 
and $(G,\sigma_2)$ are switching equivalent. (We note a sequence of switchings 
may also be realized by negating all the edges of a single edge cut.) In a very 
formal way, a signed graph is an equivalence class under the switching equivalence, 
and we sometimes use the notation $\widehat{G}$ to mean the entire class. 

It was proved by Zaslavsky~\cite{zav82b} that two signatures of $G$ are switching 
equivalent if and only if they define exactly the same set of negative (or positive) 
cycles. It is easy to conclude that a uni-balanced signed graph is switching equivalent 
to a signed graph with all edges and loops at least blue, and an anti-uni-balanced 
signed graph is switching equivalent to a signed graph with all edges and loops at
least red.

We now consider homomorphisms of signed graphs. Since signed graphs
$\widehat{G}, \widehat{H}$ can be viewed as equivalence classes, a
homomorphism of signed graphs $\widehat{G}$ to $\widehat{H}$ should be a
homomorphism of one representative $(G,\sigma)$ of $\widehat{G}$ to one
representative $(H,\pi)$ of $\widehat{H}$. It is easy to see that this
definition can be simplified by prescribing any fixed representative $(H,
\pi)$ of $\widehat{H}$. In other words, we now consider mapping all possible
representatives $(G,\sigma')$ of $\widehat{G}$ to one fixed representative $(H,\pi)$ 
of $\widehat{H}$. At this point, a homomorphism $f$ of one concrete $(G,\sigma')$
to $(H,\pi)$ is just a homomorphism of the underlying graph $G$ to the underlying graph $H$
preserving the edge colours. Since there are multiple edges, we can either
consider $f$ to be a mapping of vertices to vertices and edges to edges,
preserving vertex-edge incidences and edge-colours, as in~\cite{rezazasla}, or
simply state that blue edges map to edges that are at least blue, red edges map to
edges that are at least red, and bicoloured edges map to bicoloured edges. 
Formally, we state it as follows. 

\begin{definition}\label{ssim}
We say that a mapping $f\colon V(G) \to V(H)$ is a {\em homomorphism} of the signed
graph $(G,\sigma)$ to the signed graph $(H,\pi)$, written as $f\colon (G,\sigma)
\to (H,\pi)$, if there exists a signed graph $(G,\sigma')$, switching
equivalent to $(G,\sigma)$, such that whenever the edge $uv$ is at least positive 
in $(G,\sigma')$, then $f(u)f(v)$ is an edge that is at least positive in $(H,\pi)$, and 
whenever the edge $uv$ is at least negative in $(G,\sigma')$, then $f(u)f(v)$ is an 
edge that is at least negative in $(H,\pi)$. 
\end{definition}

There is an equivalent alternative definition (see~\cite{rezazasla}). A
homomorphism of the signed graph $(G,\sigma)$ to the signed graph $(H,\pi)$ is
a homomorphism $f$ of the underlying graph $G$ to the underlying graph $H$, which maps bicoloured
edges of $(G,\sigma)$ to bicoloured edges of $(H,\pi)$, and which for any
closed walk $W$ in $(G,\sigma)$ with only unicoloured edges for which the
image walk $f(W)$ has also only unicoloured edges, the sign of $f(W)$ in $(H,
\pi)$ is the same as the sign of $W$ in $(G,\sigma)$. (In other words, negative
closed walks map to negative closed walks and positive closed walks map to
positive closed walks.) This definition does
not require switching the input graph before mapping it. The equivalence of
the two definitions follows from the theorem of Zaslavsky~\cite{zav82b} cited
above. That result is constructive, and the actual switching required to
produce the switching equivalent signed graph $(G,\sigma')$ can be found 
in polynomial time~\cite{rezazasla}. 

We deduce the following fact.

\begin{lemma}\label{alter}
Suppose $(G,\sigma)$ and $(H,\pi)$ are signed graphs, and $f$ is a mapping of
the vertices of $G$ to the vertices of $H$. Then $f$ is a homomorphism of the
signed graph $(G,\sigma)$ to the signed graph $(H,\pi)$ if and only if $f$ is
a homomorphism of the underlying graph $G$ to the underlying graph $H$, which
moreover maps bicoloured edges of $(G,\sigma)$ to bicoloured edges of
$(H,\pi)$, and for any closed walk $W$ in $(G,\sigma)$ with only unicoloured
edges for which the image walk $f(W)$ has also only unicoloured edges, the
signs of $W$ and $f(W)$ are the same.
\end{lemma}

Note that each negative closed walk contains a negative cycle, and in
particular an irreflexive tree $(H,\pi)$ has no negative closed walks except for those using bicoloured edges.
Thus if
$(H,\pi)$ is an irreflexive tree, then the condition simplifies to having no
negative cycle of $(G,\sigma)$ mapped to unicoloured edges in $(H,\pi)$
(because the image would be a positive closed walk). 
For reflexive trees, the condition requires that no negative cycle of $(G,\sigma)$ 
maps to a positive closed walk in $(H,\pi)$, and no positive 
cycle of $(G,\sigma)$ maps to a negative closed walk.

For our purposes, the simpler Definition~\ref{ssim} is sufficient. Note that
whether an edge is unicoloured or bicoloured is independent of switching, and
that a homomorphism can map a unicoloured edge or loop in $\widehat{G}$ to a
bicoloured edge or loop in $\widehat{H}$ but not conversely.

Let $\widehat{H}$ be a fixed signed graph. The {\em homomorphism problem}
$\textsc{S-Hom}(\widehat{H})$ takes as input a signed graph $\widehat{G}$ and
asks whether there exists a homomorphism of $\widehat{G}$ to $\widehat{H}$.
The formal definition of the list homomorphism problems for signed graphs is
very similar.

\begin{definition}
Let $\widehat{H}$ be a fixed signed graph. The {\em list homomorphism problem}
\textsc{List-S-Hom}$(\widehat{H})$ takes as input a signed graph $\widehat{G}$
with lists $L(v) \subseteq V(H)$ for every $v \in V(G)$, and asks whether
there exists a homomorphism $f$ of $\widehat{G}$ to $\widehat{H}$ such that
$f(v) \in L(v)$ for every $v \in V(G)$.
\end{definition}

We note that when $\widehat{H}$ and $\widehat{H}'$ are switching equivalent
signed graphs, then any homomorphism of an input signed graph $\widehat{G}$ to
$\widehat{H}$ is also a homomorphism to $\widehat{H}'$, and therefore the
problems $\textsc{S-Hom}(\widehat{H})$ and $\textsc{S-Hom}(\widehat{H}')$, as
well as the problems \textsc{List-S-Hom}$(\widehat{H})$ and
\textsc{List-S-Hom}$(\widehat{H}')$, are equivalent.

We call a signed graph $\widehat{H}$ {\em connected} if the underlying graph
$H$ is connected. We call $\widehat{H}$ {\em reflexive} if each vertex of $H$
has a loop, and {\em irreflexive} if no vertex has a loop. We call $\widehat{H}$ 
a {\em signed tree} if $H$, with any existing loops removed, is a tree.

We may assume that the target signed graph $\widehat{H}$ is connected.
This implies no loss of generality for list homomorphism problems, as each component
of an input signed graph $\widehat{G}$ can only be mapped to one component of a 
target signed graph $\widehat{H}$.

\section{More background and connections to constraint satisfaction}\label{back}

We now briefly introduce the constraint satisfaction problems, in the format
used in~\cite{fv}. A~{\em relational system} $G$ consists of a set $V(G)$ of
vertices and a family of relations $R_1, R_2, \dots, R_k$ on $V(G)$. Assume
$G$ is a relational system with relations $R_1, R_2, \dots, R_k$ and $H$ a
relational system with relations $S_1, S_2, \dots, S_k$, where the arity of
the corresponding relations $R_i$ and $S_i$ is the same for all $i=1, 2,
\dots, k$. A {\em homomorphism} of $G$ to $H$ is a mapping $f\colon V(G) \to V(H)$
that preserves all relations, i.e., satisfies $(v_1,v_2,\dots) \in R_i
\implies (f(v_1),f(v_2),\dots) \in S_i$, for all $i=1, 2, \dots, k$. The {\em
constraint satisfaction problem} with fixed template $H$ asks whether or not
an input relational system $G$, with the same arities of corresponding
relations as $H$, admits a homomorphism to $H$.

Note that when $H$ has a single relation $S$, which is binary and symmetric,
then we obtain the graph homomorphism problem referred to at the beginning of
Section \ref{sone}. When $H$ has a single relation $S$, which is an
arbitrary binary relation, we obtain the {\em digraph homomorphism
problem}~\cite{gary} which is in a certain sense~\cite{fv} as difficult to
classify as the general constraint satisfaction problem. When $H$ has two
relations $+, -$, then we obtain a problem that is superficially similar to
the homomorphism problem for signed graphs, except that switching is not
allowed. This problem is called the {\em edge-coloured graph homomorphism
problem}~\cite{rickphd}, and it turns out to be similar to the digraph
homomorphism problem in that it is difficult to classify~\cite{BFHN}. On the
other hand, the homomorphism problem for signed
graphs~\cite{BFHN,dichotomy,FN14}, seems easier to classify, and exhibits a
dichotomy similar to the graph dichotomy classification, see
Theorem~\ref{thm:dichotomy}.

List homomorphism problems are also special cases of constraint satisfaction
problems, as lists can be replaced by unary relations. Consider first the case
of graphs. Suppose $H$ is a fixed graph, and form the relational system $H^{\#}$
with vertices $V(H)$ and the following relations: one binary relation $E(H)$
(this is a symmetric relation corresponding to the undirected edges of the
graph $H$), and $2^{|V(H)|}-1$ unary relations $R_X$ on $V(H)$, each
consisting of a different non-empty subset $X$ of $V(H)$. The constraint
satisfaction problem with template $H^{\#}$ has inputs $G$ with a symmetric
binary relation $E(G)$ (a graph) and unary relations $S_X, X \subseteq V(H)$,
and the question is whether or not a homomorphism exists. If a vertex $v \in
V(G)$ is in the relation $S_X$ corresponding to $R_X$, then any mapping 
preserving the relations must map $v$ to a vertex in $X$; thus imposing the 
relation $S_X$ on $v \in V(G)$ amounts to setting $L(v)=X$. Therefore the 
list homomorphism problem for the graph $H$ is formulated as the
constraint satisfaction problem for the template $H^{\#}$.
 
Such a translation is also possible for homomorphism of signed graphs. Brewster and Graves 
introduced a useful construction. The {\em switching graph} $(H^+,\pi^+)$ has two 
vertices $v_1,v_2$ for each vertex $v$ of $(H,\pi)$, and each edge $vw$ of $(H,\pi)$ gives
rise to edges $v_1w_1, v_2w_2$ of colour $\pi(vw)$ and edges $v_1w_2, v_2w_1$
of the opposite colour. (This definition applies also for loops, i.e., when
$v=w$.) Then each homomorphism of the signed graph $(G,\sigma)$ to the signed
graph $(H,\pi)$ corresponds to a homomorphism of the edge-coloured graph
$(G,\sigma)$ to the edge-coloured graph $(H^+,\pi^+)$ and conversely. For list
homomorphisms of signed graphs, we can use the same transformation, modifying
the lists of the input signed graph. If $(G,\sigma)$ has lists $L(v), v \in
V(G)$, then the new lists $L^+(v), v \in V(G),$ are defined as follows: for
any $x \in L(v)$ with $v \in V(G)$, we place both $x_1$ and $x_2$
in $L^+(v)$. It is easy to see that the signed graph $(G,\sigma)$ has a list
homomorphism to the signed graph $(H,\pi)$ with respect to the lists $L$ if
and only if the edge-coloured graph $(G,\sigma)$ has a list homomorphism to
the edge-coloured graph $(H^+,\pi^+)$ with respect to the lists $L^+$. The new
lists $L^+$ are {\em symmetric sets in $H^+$}, meaning that for any $x \in
V(H), v \in V(G)$, we have $x_1 \in L^+(v)$ if and only if we have $x_2 \in
L^+(v)$. Thus we obtain the list homomorphism problem for the edge-coloured
graph $(H^+,\pi^+)$, restricted to input instances $(G,\sigma)$ with lists $L$
that are symmetric in $H^+$. As above, we can transform this list homomorphism 
problem for the edge-coloured graph $(H^+,\pi^+)$, to a constraint satisfaction 
problem. The details are similar to the construction of $H^{\#}$, except this time 
the new template $(H^+,\pi^+)^*$ is obtained by adding unary relations $R_X = X$
only for sets $X \subseteq V(H^+)$ that are symmetric in $H^+$. 

We conclude that 
our problems \textsc{List-S-Hom}$(\widehat{H})$ fit into the general constraint
satisfaction framework, and therefore it follows from~\cite{bula,zhuk} that
dichotomy holds for problems \textsc{List-S-Hom}$(\widehat{H})$. We therefore
ask which problems \textsc{List-S-Hom}$(\widehat{H})$ are polynomial-time
solvable and which are NP-complete.

The solution of the Feder-Vardi dichotomy conjecture involved an
algebraic classification of the complexity pioneered by Jeavons~\cite{jeav}. A
key role in this is played by the notion of a polymorphism of a relational
structure $H$. If $H$ is a digraph, then a {\em polymorphism} of $H$ is a
homomorphism $f$ of some power $H^t$ to $H$, i.e., a function $f$ that assigns
to each ordered $t$-tuple $(v_1,v_2,\ldots,v_t)$ of vertices of $H$ a vertex
$f(v_1,v_2,\ldots,v_t)$ such that two coordinate-wise adjacent tuples obtain
adjacent images. For general templates, all relations must be similarly
preserved. A polymorphism of order $t=3$ is a {\em majority} if
$f(v,v,w)=f(v,w,v)=f(w,v,v)=v$ for all $v, w$. A {\em Siggers polymorphism} is
a polymorphism of order $t=4$, if $f(a,r,e,a)=f(r,a,r,e)$ for all $a, r, e$.
One formulation of the dichotomy theorem proved by Bulatov~\cite{bula} and
Zhuk~\cite{zhuk} states that the constraint satisfaction problem for the
template $H$ is polynomial-time solvable if $H$ admits a Siggers polymorphism,
and is NP-complete otherwise. Majority polymorphisms are less powerful, but it
is known~\cite{fv} that if $H$ admits a majority then the constraint
satisfaction problem for the template $H$ is polynomial-time solvable.
Moreover, it was shown in~\cite{feder2003bi} that a graph $H$ is a bi-arc
graph if and only if the associated relational system $H^*$ admits a majority
polymorphism. Thus the list homomorphism problem for a graph $H$ with possible
loops is polynomial-time solvable if $H^*$ admits a majority polymorphism, and
is NP-complete otherwise. It was observed in \cite{hyobin} that this is not true
for signed graphs. 

There is a convenient way to think of polymorphisms $f$ of the relational
system $(H^+,\pi^+)^*$. A mapping $f$ is a polymorphism of $(H^+,\pi^+)^*$ 
if and only if it is a polymorphism of the edge-coloured graph $(H^+,\pi^+)$ and if, for
any symmetric set $X \subseteq V(H^+)$, we have $x_1, x_2, \dots, x_t \in X$
then also $f(x_1,x_2,\ldots,x_t) \in X$. We call such polymorphisms of
$(H^+,\pi^+)$ {\em semi-conservative}.

We can apply the dichotomy result of~\cite{bula,zhuk} to obtain an algebraic
classification.

\begin{theorem}
For any signed graph $(H,\pi)$, the problem \textsc{List-S-Hom}$(H,\pi)$ is
polynomial-time solvable if $(H^+,\pi^+)$ admits a semi-conservative Siggers
polymorphism, and is NP-complete otherwise.
\end{theorem}

As mentioned above, one can not replace the semi-conservative Siggers 
polymorphism by a semi-conservative majority polymorphism \cite{hyobin}. 
We focus in this paper on seeking a graph theoretic classification, at least 
for some classes of signed graphs.

\section{Basic facts} \label{sec:basicfacts}

We first mention the dichotomy classification of the problems \textsc{S-Hom}$(\widehat{H})$
from \cite{dichotomy}. A subgraph $\widehat{G}$ of the signed graph $\widehat{H}$ is the 
\emph{s-core} of $\widehat{H}$ if there is homomorphism  $f: \widehat{H} \to \widehat{G}$, and
every homomorphism $\widehat{G} \to \widehat{G}$ is a bijection on $V(G)$. The letter $s$ stands
for {\em signed}. It is again easy to see that the $s$-core is unique up to isomorphism and
switching equivalence.

\begin{theorem}{\cite{dichotomy}}\label{thm:dichotomy}
The problem \textsc{S-Hom}$(\widehat{H})$ is polynomial-time solvable
if the s-core of $\widehat{H}$ has at most two edges, and is NP-complete otherwise.
\end{theorem}

When the signature $\pi$ has all edges positive, the problem \textsc{S-Hom}$(H,\pi)$ 
is equivalent to the unsigned graph homomorphism problem, and the s-core of $(H,\pi)$ 
is just the core of $H$. To compare Theorem~\ref{thm:dichotomy} with the graph dichotomy 
theorem of~\cite{hn} as discussed at the beginning Section~\ref{sone}, we observe that
the core of a graph cannot have exactly two edges, as a core must be either a single vertex 
(possibly with a loop), or a single edge, or a graph with at least three edges. Thus Theorem 
\ref{thm:dichotomy} is stronger than the graph dichotomy theorem from~\cite{hn}, which states 
that the graph homomorphism problem to $H$ is polynomial-time solvable if the core of $H$ 
has at most one edge and is NP-complete otherwise. (However, we note that the proof of
Theorem~\ref{thm:dichotomy} in~\cite{dichotomy} uses the graph dichotomy theorem~\cite{hn}.)

Observe that an instance of the problem $\textsc{S-Hom}(\widehat{H})$ can be
also viewed as an instance of $\textsc{List-S-Hom}(\widehat{H})$ with all
lists $L(v)=V(H)$, therefore if $\textsc{S-Hom}(\widehat{H})$ is NP-complete,
then so is $\textsc{List-S-Hom}(\widehat{H})$. Moreover, if $\widehat{H}'$ is
an induced subgraph of $\widehat{H}$, then any instance of
$\textsc{List-S-Hom}(\widehat{H}')$ can be viewed as an instance of
$\textsc{List-S-Hom}(\widehat{H})$ (with the same lists), therefore if the
problem $\textsc{List-S-Hom}(\widehat{H}')$ is NP-complete, then so is the
problem $\textsc{List-S-Hom}(\widehat{H})$. This yields the NP-completeness of
$\textsc{List-S-Hom}(\widehat{H})$ for all signed graphs $(\widehat{H})$ that
contain an induced subgraph $\widehat{H}'$ whose s-core has more than two
edges. Furthermore, when the signed graph $\widehat{H}$ is uni-balanced, then we 
may assume that all edges are at least blue, and the list homomorphism
problem for $H$ can be reduced to $\textsc{List-S-Hom}(\widehat{H})$. In particular, 
we emphasize that  {\em $\textsc{List-S-Hom}(\widehat{H})$ is NP-complete if $\widehat{H}$ 
is a uni-balanced signed graph (or, by a symmetric argument, an anti-uni-balanced signed 
graph), and the underlying graph $H$ is not a bi-arc graph}~\cite{feder2003bi}.

Next we focus on the class of signed graphs that have no bicoloured loops and no bicoloured 
edges. In this case, the following simple dichotomy describes the classification. (This result 
was previously announced in~\cite{bordeaux}.) It follows from our earlier remarks that these
signed graphs are balanced if and only if they are uni-balanced, and similarly they are anti-balanced 
if and only if they are anti-uni-balanced.

\begin{theorem} \label{thm:bordeaux}
Suppose $\widehat{H}$ is a connected signed graph without bicoloured loops and edges.
If the underlying graph $H$ is a bi-arc graph, and $\widehat{H}$ is balanced or 
anti-balanced, then the problem $\textsc{List-S-Hom}(\widehat{H})$ is polynomial-time 
solvable. Otherwise, the problem is NP-complete.
\end{theorem}

\begin{proof}
The polynomial cases follow from Feder et al.~\cite{feder2003bi}, by the following argument.
Suppose $\widehat{H}$ is balanced; we may assume all edges are blue. In~\cite{zav82b}, 
there is a polynomial-time algorithm to decide if the input signed graph $\widehat{G}$ is 
balanced. If it is not balanced, there is no homomorphism of $\widehat{G}$ to $\widehat{H}$.
Otherwise, we may assume that $\widehat{G}$ has also all edges blue and hence there is a
homomorphism of $\widehat{G}$ to $\widehat{H}$ if and only if there is a homomorphism of $G$
to $H$. Since $H$ is a bi-arc graph, this can be decided in polynomial time by the algorithm in
\cite{feder2003bi}. The argument is similar if $\widehat{H}$ is anti-balanced. Otherwise, $\widehat{H}$
contains a cycle which cannot be switched to a blue cycle and a cycle which cannot be switched 
to a red cycle, in which case the s-core of $\widehat{H}$ contains at least three edges. 
(This is true even if the cycles are just loops.)
\qed \end{proof}

We have observed that \textsc{List-S-Hom}$(\widehat{H})$ is NP-complete if the
s-core of $\widehat{H}$ has more than two edges. Thus we will focus on signed
graphs $\widehat{H}$ whose s-cores have at most two edges. This is not as simple
as it sounds, as there are many complex signed graphs with this property, including, 
for example, all irreflexive bipartite signed graphs that contain a bicoloured edge, 
and all signed graphs that contain a bicoloured loop. That these cases are not easy
underlines the fact that the assumptions in Theorem~\ref{thm:bordeaux} cannot be
weakened without significant new breakthroughs. Consider, for example, allowing 
bicoloured edges but not bicoloured loops. In this situation, we may focus on the 
case when there is a bicoloured edge (else Theorem~\ref{thm:bordeaux} applies),
and so if there is any loop at all, the s-core would have more than two edges. Thus
we consider irreflexive signed graphs. The s-core is still too big if the underlying
graph has an odd cycle. So in this case it remains to classify the irreflexive bipartite 
signed graphs that contain a bicoloured edge. Even this case is complex. We explore
homomorphisms to irreflexive bipartite signed graphs in a companion paper.

In this paper we focus on $\textsc{List-S-Hom}(\widehat{H})$ when the underlying 
graph of $\widehat{H}$ is a tree with possible loops. We have treated the special cases
of reflexive and irreflexive trees in~\cite{bordeaux} and in the conference version of this paper~\cite{mfcs};
in the general case presented here, the polynomial algorithms are much more involved
and technical, and the structure of these trees turns out to be surprisingly complex.
Nevertheless, they seem to suggest that nice characterizations may be possible, in
analogy to bi-arc trees \cite{feder2007bitr}, see also \cite{feder2003bi}.

We now introduce our basic tool for proving NP-completeness.

\begin{definition}
Let $(U,D)$ be two walks in $\widehat{H}$ of equal length, say $U$, with
vertices $u = u_0,u_1, \ldots,u_k=v$ and $D$, with vertices $u = d_0,d_1,
\ldots,d_k=v$. We say that $(U,D)$ is a \emph{chain}, provided $uu_1,
d_{k-1}v$ are unicoloured edges and $ud_1, u_{k-1}v$ are bicoloured edges, and
for each~$i$, $1 \leq i \leq k-2$, we have
\begin{enumerate}
\item both $u_iu_{i+1}$ and $d_id_{i+1}$ are edges of $\widehat{H}$ 
while $d_iu_{i+1}$ is not an edge of $\widehat{H}$, or
\item both $u_iu_{i+1}$ and $d_id_{i+1}$ are bicoloured edges of $\widehat{H}$ 
while $d_iu_{i+1}$ is not a bicoloured edge of~$\widehat{H}$.
\end{enumerate}
\end{definition}

\begin{theorem}\label{thm:chain}
If a signed graph $\widehat{H}$ contains a chain, then
$\textsc{List-S-Hom}(\widehat{H})$ is NP-complete.
\end{theorem}

\begin{proof}
Suppose that $\widehat{H}$ has a chain $(U,D)$ as specified above. We shall
reduce from {\sc Not-All-Equal SAT}. (Each clause has three unnegated
variables, and we seek a truth assignment in which at least one variable is
true and at least one is false, in each clause.) For each clause $x \lor y
\lor z$, we take three vertices $x,y,z$, each with list $\{u\}$, and three
vertices $x',y',z'$, each with list $\{v\}$. For the triple $x,y,z$, we add
three new vertices $p(x,y)$, $p(y,z)$, and $p(z,x)$, each with list
$\{u_1,d_1\}$, and for the triple $x', y', z'$, we add three new vertices
$p(x',y'), p(y',z'), p(z',x')$, each with list $\{u_{k-1},d_{k-1}\}$. We
connect these vertices as follows:
\begin{itemize}
\item $p(x,y)$ adjacent to $x$ by a red edge and to $y$ by a blue edge,
\item $p(y,z)$ adjacent to $y$ by a red edge and to $z$ by a blue edge,
\item $p(z,x)$ adjacent to $z$ by a red edge and to $x$ by a blue edge.
\end{itemize}
Analogously, the hexagon $x', p(x',y'),  y', p(y',z'), z', p(z',x')$ will also
be alternating in blue and red colours, with (say) $p(x',y')$ adjacent to $x'$
by a red edge.

Moreover, we join each pair of vertices $p(x,y)$ and $p(x',y')$ by a separate
path $P(x,y)$ with $k-1$ vertices, say $p(x,y)=a_1, a_2, \ldots, a_{k-2},
a_{k-1}=p(x',y')$, where $a_{i}$ has list $\{u_i, d_i\}$ and the edge
$a_ia_{i+1}$ is blue unless both $u_iu_{i+1}$ and $d_id_{i+1}$ are bicoloured,
in which case $a_ia_{i+1}$ is also bicoloured. Paths $P(z,x)$ and $P(y,z)$ are
defined analogously. See Figure~\ref{fig:chain} for an illustration.

\begin{figure}[b]
\centering
\includegraphics[scale=1]{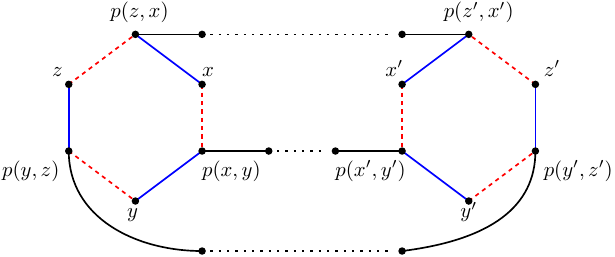}
\caption{The clause gadget for clause $(x \lor y \lor z)$ in Theorem~\ref{thm:chain}.}
\label{fig:chain}
\end{figure}

We observe for future reference that the path $x, p(x,y)=a_1, a_2, \ldots,
a_{k-2},$ $a_{k-1}=p(x',y'), x'$, when considered by itself, admits a list
homomorphism both to $U$ and to $D$. (To see this invoke Zaslavsky's theorem 
characterizing switching equivalent signatures, and use the fact that both 
$U$ and $D$ contain a bicoloured edge.)  Further we note there is no list 
homomorphism to any other subgraph of $U \cup D$ where $p(x,y)$ maps 
to $d_1$ and $p(x',y')$ maps to $u_{k-1}$.  (This follows from the conditions 
in the definition of chain.) 

If $x$ occurs in several clauses, we now have several vertices corresponding 
to $x$. We link all these occurrences by a new vertex $p(x)$ with the list $\{u_1\}$, 
and joined by blue edges to all the occurrences of $x$. Since the edge
$uu_1$ is unicoloured, this will ensure that all occurrences of the vertex $x$ are 
switched or no occurrence of $x$ is switched.  

We denote the resulting graph $(G,\sigma)$. We now claim that this instance of
{\sc {\sc Not-All-Equal SAT}} is satisfiable if and  only if $(G,\sigma)$
admits a list homomorphism to $(H,\pi)$.

Let $\widehat{G}(x,y,z)$ denote the subgraph of $(G,\sigma)$ induced by
$P(x,y), P(z,x), P(y,z)$ and $x, y, z, x', y', z'$. We claim that\\ {\em (i)
any list homomorphism of $\widehat{G}(x,y,z)$ to $U \cup D$ must switch at
either one or two of the  vertices $x, y, z$}, and that \\ {\em (ii) there are
list homomorphisms of $\widehat{G}(x,y,z)$ to $U \cup D$ that switch at any 
one or any two of the vertices $x, y, z$.}

Once this claim is proved, we can associate with every truth assignment a list
homomorphism of $\widehat{G}(x,y,z)$ to $U \cup D$ where a vertex
corresponding to a variable is switched if and only if that variable is true,
and conversely, setting a variable true if its corresponding vertex was
switched in the list homomorphism.
Since we have ensured that all occurrences of a variable are switched or all 
are not switched, we conclude that all occurrences of a variable take on the 
same truth value.

We now prove (i). Since the lists are so restrictive, any list homomorphism is
fully described by what happens to the paths $P(x,y)$, $P(z,x)$, $ P(y,z)$,
and whether or not the vertices corresponding to $x, y, z$ (and $x', y', z'$)
are switched. Note that $U$ begins with a unicoloured edge and $D$ ends with a
unicoloured edge. If neither $x$ nor $y$ or both $x$ and $y$
are switched, the edges $xp(x,y)$ and $p(x,y)y$ are different colours, and 
in any list homomorphism of $\widehat{G}(x,y,z)$ to $U \cup D$ we must map
$P(x,y)$ to $D$.  In particular, $x'p(x',y')$ and $p(x',y')y'$ map to a unicoloured 
edge and must have the same colour.  Thus, if none or all of the vertices $x, y, z$ 
were switched, then the hexagon $x', p(x',y'),  y', p(y',z'), z', p(z',x')$ has an
even number of red and even number of blue edges, which is impossible. (It
started with an odd number of each.) 

For (ii), it remains to show that one or two of the vertices $x, y, z$ can be
switched under a list homomorphism of $\widehat{G}(x,y,z)$ to $U \cup D$. Suppose
first that only one was switched; by symmetry assume it was $x$ (so $y$ and
$z$ were not switched). Now edges $x, p(x,y)$ and $p(x,y), y$ have the same
colour, and $z, p(z,x)$ and $p(z,x), x$ have the same colour. By the above
observation, we can map $P(x,y)$ and $P(z,x)$ to $U$, and map $P(y,z)$ to $D$.
Note that the switchings necessary for these list homomorphisms affect
disjoint sets of vertices (the paths $P(x,y), P(y,z), P(z,x)$), so the
observation applies. If two vertices, say $y$ and $z$ were switched, the
argument is almost the same and we omit it.
\qed \end{proof}

\section{Irreflexive trees} \label{sec:irreflexive}

In this section, $\widehat{H}$ will always be an irreflexive tree. As trees do
not have any cycles, $\widehat{H}$ is trivially uni-balanced, and hence we may 
assume that all edges are at least blue.

\begin{lemma} \label{lem:tripleclaw}
If the underlying graph $H$ contains the graph $F_1$ 
in Figure~\ref{fig:tripleclaw}, then $\textsc{List-S-Hom}(\widehat{H})$ is NP-complete.
\end{lemma}

\begin{proof}
If the underlying graph $H$ contains the graph $F_1$ in
Figure~\ref{fig:tripleclaw}, then $H$ is not a bi-arc graph
by~\cite{feder2003bi}, whence $\textsc{List-S-Hom}(\widehat{H})$ is
NP-complete by the remarks following Theorem~\ref{thm:dichotomy}. 
\qed \end{proof}

\begin{figure}[b]
\centering
\includegraphics[scale=1]{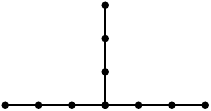}
\caption{The subgraph $F_1$.}
\label{fig:tripleclaw}
\end{figure}

\begin{lemma}\label{pejsek}
If $\widehat{H}$ contains one of the signed
graphs in family $\cal F$ from Figure~\ref{fig:forbgraphsirref} as an induced
subgraph, then $\textsc{List-S-Hom}(\widehat{H})$ is NP-complete
\end{lemma}

\begin{proof}
For each signed tree in the family $\cal F$ we specify a chain either directly in 
Figure~\ref{fig:forbgraphsirref}, or (in case a)) indirectly by reference to a more general
situation addressed in Figure~\ref{fig:mixed_merged}. By Theorem \ref{thm:chain}, these
signed trees yield NP-complete list homomorphism problems. Thus any signed graph $\widehat{H}$
that contains one of them as an induced subgraph has also the problem $\textsc{List-S-Hom}(\widehat{H})$
NP-complete.
\qed \end{proof}

\begin{figure}[t]
\centering
\includegraphics[width=\textwidth]{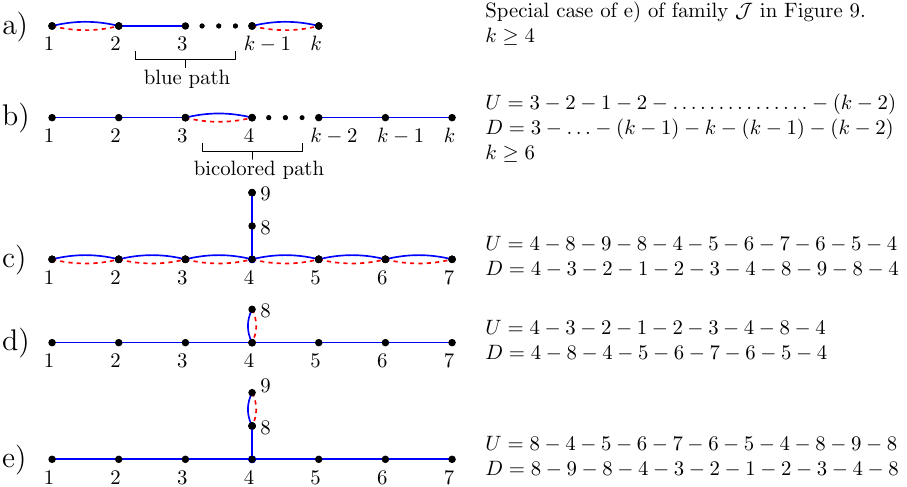}
\caption{The family $\cal F$ of signed irreflexive trees with NP-complete problems.}
\label{fig:forbgraphsirref}
\end{figure}

An irreflexive tree $H$ is a {\em $2$-caterpillar} if it contains a path $P =
v_1v_2 \ldots v_k$, such that each vertex of $H$ is either on $P$, or is a
child of a vertex on $P$, or is a {\em grandchild} of a vertex on $P$, i.e., 
is adjacent to a child of a vertex on $P$. We also say that $H$
is a $2$-caterpillar {\em with respect to the spine $P$}. (Note that the same
tree $H$ can be a $2$-caterpillar with respect to different spines $P$.) In
such a situation, let $T_1, T_2, \dots, T_{\ell}$ be the connected components
of $H \setminus P$. Each $T_i$ is a star adjacent to a unique vertex $v_j$ on
$P$. The tree $T_i$ together with the edge joining it to $v_j$ is called a
{\em rooted subtree} of $H$ (with respect to the spine $P$), and is considered
to be rooted at $v_j$. Note that there can be several rooted subtrees with the
same root vertex $v_j$ on the spine, but each rooted subtree at $v_j$ contains
a unique child of $P$ (and possibly no grandchildren, or possibly several
grandchildren).

We also use the term {\em $2$-caterpillar} for any signed graph to mean that the 
underlying graph, with loops removed, is a $2$-caterpillar.

If $H$ is a $2$-caterpillar with respect to the spine $P$, and additionally
the bicoloured edges of $\widehat{H}$ form a connected subgraph, and there
exists an integer $d$, with $1 \leq d \leq k$, such that:
\begin{itemize}
\item all edges on the path $v_1v_2 \ldots v_d$ are bicoloured, and all edges 
on the path $v_dv_{d+1} \ldots v_k$ are blue,
\item the edges of all subtrees rooted at $v_1, v_2, \ldots, v_{d-1}$ are 
bicoloured, except possibly edges incident to leaves, and
\item the edges of all subtrees rooted at $v_{d+1}, \ldots, v_k$ are 
all blue,
\end{itemize}
\noindent then we call $\widehat{H}$ a {\em good $2$-caterpillar} with respect 
to $P = v_1v_2 \ldots v_k$.

The vertex $v_d$ is called the {\em dividing vertex} of $\widehat{H}$. Note
that the subtrees rooted at $v_d$ are not limited by any condition except the
connectivity of the subgraph formed by the bicoloured edges. A typical example
of a good $2$-caterpillar is depicted in Figure~\ref{fig:example_good_irref}.

\begin{figure}[b]
\centering
\includegraphics[scale=1.15]{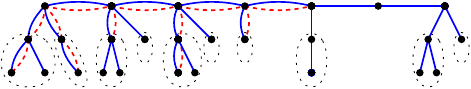}
\caption{An example of a good $2$-caterpillar.}
\label{fig:example_good_irref}
\end{figure}

\begin{lemma} \label{lem:forbirref}
Let $\widehat{H}$ be an irreflexive signed tree. Then $\widehat{H}$ is a good 
$2$-caterpillar if and only if it does not contain any of the graphs from family ${\cal F}$
in Figure~\ref{fig:forbgraphsirref} as an induced subgraph, and the underlying 
graph $H$ does not contain the graph $F_1$ in Figure~\ref{fig:tripleclaw}.
\end{lemma}

\begin{proof}
It is easy to check that none of the depicted signed graphs admits a suitable
spine, and hence they are not good $2$-caterpillars. So assume $\widehat{H}$
does not contain any of the graphs in Figure~\ref{fig:forbgraphsirref} as an
induced subgraph, and the underlying graph $H$ does not contain the graph
$F_1$ in Figure~\ref{fig:tripleclaw} as a subgraph.
If $H$ is not a $2$-caterpillar with respect to any spine, then the underlying
graph $H$ contains the tree in Figure~\ref{fig:tripleclaw}. The bicoloured
edges of $\widehat{H}$ induce a connected subgraph, since there is no subgraph
of type a) from Figure~\ref{fig:forbgraphsirref}. 
Similarly, the unicoloured edges between two non-leaf vertices of $\widehat{H}$
induce a connected subgraph, since there is no subgraph of type b) in 
Figure~\ref{fig:forbgraphsirref}. 
Let $\widehat{H'}$ denote the signed tree obtained from $\widehat{H}$ by removing all
vertices that are leaves incident with a unicoloured edge. Then we can conclude
that there is a vertex $v_d$ that separates unicoloured and bicoloured edges in 
$\widehat{H'}$.

The absence of classes b) and c) from Figure~\ref{fig:forbgraphsirref} 
also ensures that, there is a suitable spine $P = v_1v_2 \ldots v_k$ with
bicoloured edges on $v_1v_2 \ldots v_d$ and blue edges on $v_dv_{d+1} \ldots
v_k$, and with all subtrees of height two rooted at $v_1, \ldots, v_{d-1}$
attached to $P$ with a bicoloured edge. (For this, we note that in the case c),
as long as the edges $34$ and $45$ are bicoloured, the subtree still yields 
an NP-complete problem even with any of the edges $12, 23, 56, 67$ 
unicoloured.) The absence of graphs d) and e) in Figure~\ref{fig:forbgraphsirref} 
ensures that all subtrees rooted at $v_{d+1}, \ldots, v_k$ have all edges blue.
\qed \end{proof}

\begin{theorem} \label{thm:main_irref}
Let $\widehat{H}$ be an irreflexive tree. If $\widehat{H}$ is a good
$2$-caterpillar, then $\textsc{List-S-Hom}(\widehat{H})$ is polynomial-time
solvable. Otherwise, $H$ contains a copy of $F_1$, or $\widehat{H}$ contains 
one of the signed graphs in family $\cal F$ as an induced subgraph, and 
the problem is NP-complete.
\end{theorem}

The second claim follows from Lemmas~\ref{lem:tripleclaw}, \ref{pejsek} and~\ref{lem:forbirref}. 
We prove the first claim in a sequence of lemmas. Suppose that $\widehat{H}$
is a good $2$-caterpillar with respect to the spine $P=v_1v_2 \ldots v_k$.
Since $H$ is bipartite, we may distinguish its vertices as {\em black} and
{\em white}. We may assume that the input signed graph $\widehat{G}$ is
connected and bipartite, and the lists of the black vertices of $G$ contain
only black vertices of $H$, and similarly for white vertices. (Since $G$ is
connected, there are only two possible assignments of black and white colours
to its vertices, and we consider each separately.)

We distinguish four types of rooted subtrees of $\widehat{H}$ with respect to
the spine $P$. 
\begin{itemize}
\item Type $T_1$: a bicoloured edge $v_ix$;
\item Type $T_2$: a bicoloured edge $v_ix$, bicoloured edges $xz_j$ for a set 
of vertices $z_j$, and blue edges $xt_j$ for another set of vertices $t_j$;
\item Type $T_3$: a blue edge $v_ix$ and blue edges $xt_j$ for a set of vertices $t_j$; and
\item Type $T_4$: a blue edge $v_ix$.
\end{itemize}
In the types $T_2$ and $T_3$ we assume that they are not of type $T_1$ or
$T_4$, i.e., that at least some $z_j$ or $t_j$ exist; but we allow in $T_2$
either the set of $z_j$ or the set of $t_j$ to be empty.

Recall that we assume that all edges of $\widehat{H}$ are at least blue.
Since the underlying graph $H$ is bipartite, we have also distinguished its
vertices as black and white; we assume that $v_1$ is white. 

A {\em bipartite min ordering} of the bipartite graph $H$ is a pair $<_b,
<_w$, where $<_b$ is a linear ordering of the black vertices and $<_w$ is a
linear ordering of the white vertices, such that for white vertices $x <_w x'$
and black vertices $y <_b y'$, if $xy', x'y$ are both edges in $H$, then $xy$
is also an edge in $H$. It is known~\cite{fv} that if a bipartite graph $H$
has a bipartite min ordering, then the list homomorphism problem for $H$ can
be solved in polynomial time as follows. 
First apply the {\em arc consistency}
test, which repeatedly visits edges $xy$ and removes from $L(x)$ any vertex of
$H$ not adjacent to some vertex of $L(y)$, and similarly removes from $L(y)$
any vertex of $H$ not adjacent to some vertex of $L(x)$. After arc
consistency, if there is an empty list, no list homomorphism exists, and if
all lists are non-empty, choosing the minimum element of each list, according
to $<_b$ or $<_w$, defines a list homomorphism as required. 
We call a
bipartite min ordering of the signed irreflexive tree $\widehat{H}$ {\em
special} if for any black vertices $x, x'$ and white vertices $y, y'$, if $xy$ is
bicoloured and $xy'$ is blue, then $y <_w y'$, and if $xy$ is bicoloured and
$x'y$ is blue, then $x <_b x'$. In other words, the bicoloured neighbours of
any vertex appear before its unicoloured neighbours, both in $<_b$ and in
$<_w$.

\begin{lemma} \label{mino}
Every good $2$-caterpillar $\widehat{H}$ admits a special bipartite min ordering.
\end{lemma}

\begin{proof}
Let us first observe that any $2$-caterpillar admits a bipartite min ordering
$<_b, <_w$ with $v_1 <_w v_3 <_w v_5, \ldots$ and $v_2 <_b v_4 <_b v_6,
\ldots$, in which the vertices of each subtree rooted at a vertex $v_i$ are placed
as follows: all non-leaf children of $v_i$, as well as all leaf children of $v_i$
adjacent to $v_i$ by bicoloured edges, are ordered between $v_{i-1}$ and $v_{i+1}$,
all leaf children of $v_i$ adjacent to $v_i$ by unicoloured edges are ordered
between $v_{i+1}$ and $v_{i+3}$, and all grandchildren of $v_i$ are ordered
between $v_i$ and $v_{i+2}$. Moreover, we ensure that the order of the
grandchildren conforms to the order of the children, i.e., if a child $a$ of
$v_i$ is ordered before a child $b$ of $v_i$ then the children of $a$ are all
ordered before the children of $b$. Finally, all children of $v_i$ are ordered
after all the grandchildren of $v_{i-1}$. See Figure~\ref{fig:order_irref} for an
illustration.

\begin{figure}[b]
\centering
\includegraphics[scale=1.2]{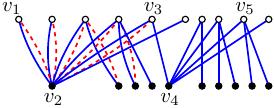}
\caption{An example of special bipartite min ordering.}
\label{fig:order_irref}
\end{figure}

It remains to ensure that the bipartite min ordering we choose is in fact a
special bipartite min ordering, i.e., that each vertex has its neighbours
joined by bicoloured edges ordered before its neighbours joined by
unicoloured edges. Therefore the subtrees rooted at each $v_i$ are handled 
as follows. We will order first the vertices of subtrees of type $T_1$, one at a 
time, then order the vertices of subtrees of type $T_2$, one at a time, then the
vertices of subtrees of type $T_3$, one at a time, and finally the vertices of 
subtrees of type $T_4$, one at a time. Each subtree of type $T_1$ consists 
of only one bicoloured edge, and we order these consecutively between $v_{i-1}$ 
and $v_{i+1}$. Next in order will come the children of $v_i$ in subtrees of type
$T_2$, still before $v_{i+1}$, and in each of these subtrees we order first the
grandchildren of $v_i$ incident to a bicoloured edge before those incident to
a unicoloured edge. We order the subtrees of type $T_3$ similarly. 
Note that by the definition of a good $2$-caterpillar, the subtrees of type $T_3$ 
can only be rooted at vertices $v_i$ with $d \leq i \leq k$. Thus, if we have a 
blue child of $v_i$ ordered before $v_{i+1}$, then $v_iv_{i+1}$ is unicoloured.
Finally, for subtrees of type $T_4$, we order their vertices (each a child of $v_i$)
right after $v_{i+1}$.
\qed \end{proof}

\begin{lemma} \label{minp}
If a signed irreflexive tree $\widehat{H}$ admits a special bipartite min 
ordering, then $\textsc{List-S-Hom}(\widehat{H})$ is polynomial-time solvable.
\end{lemma}

\begin{proof}
We describe a polynomial-time algorithm. Suppose $\widehat{G}$ is the input
signed graph; we may assume $\widehat{G}$ is connected, bipartite, and such
that the black vertices have lists with only the black vertices of
$\widehat{H}$, and similarly for the white vertices. The first step is to
perform the arc consistency test for the existence of a homomorphism of the
underlying graphs $G$ to $H$, using the special bipartite min ordering $<_b,
<_w$. We also perform the {\em bicoloured arc consistency} test, which repeatedly
visits bicoloured edges $xy$ of $G$ and removes from $L(x)$ any vertex of $H$
not adjacent to some vertex of $L(y)$ by a bicoloured edge, and similarly
removes from $L(y)$ any vertex of $H$ not adjacent to some vertex of $L(x)$ by
a bicoloured edge. If this yields an empty list, there is no list homomorphism of
the underlying graphs, and hence no list homomorphism of signed graphs.
Otherwise, the minima of all lists define a list homomorphism $f\colon G \to H$ of
the underlying graphs, by~\cite{fv}. By the bicoloured arc consistency test,
the minimum choices imply that the image of a bicoloured edge under $f$ is
also a bicoloured edge. According to Lemma~\ref{alter} and the remark
following it, $f$ is also a  list homomorphism of signed graphs unless a
negative cycle $C$ of unicoloured edges of $\widehat{G}$ maps to a closed walk
$f(C)$ of blue edges in $\widehat{H}$. Now we make use of the properties of
special bipartite min ordering to repair the situation, if possible. Note that 
the fact that we choose minimum possible values for $f$ means that we 
cannot map $C$ lower in the orders $<_b, <_w$. We consider three possible 
cases.

\begin{itemize}
\item \textit{At least one of the edges of $f(C)$ is in a subtree $T$ of type 
$T_2$ rooted at some $v_i,$ with $i \leq d$:}\\
In this case, all edges of $f(C)$ must be in $T$, since the edge of $T$
incident to $v_i$ is bicoloured. Assume without loss of generality that $v_i$
is white, $x$ is the unique child of $v_i$ in $T$, and $xt_1, \ldots, xt_m$
are the blue edges of $T$, where $x$ is black and $t_1, \ldots, t_m$ are
white. Since $f(C)$ is included in the edges $xt_1, \ldots, xt_m$ and $v_i$
precedes in $<_w$ all vertices $t_1, \ldots, t_m$, the final lists of the
white vertices in $C$ do not include $v_i$ (since we assigned the minimum
value in each list). Therefore under any homomorphism the image of the
connected graph $C$ either is included in the set of edges $xt_1, \ldots,
xt_m$, or is disjoint from this set of edges. Since we have already explored
the first possibility, we can delete the vertices $t_1, \ldots, t_m$ from the
lists of all white vertices of $C$ and repeat the arc consistency test. This
will check whether there is possibly another list homomorphism of graphs $G
\to H$, which is also a homomorphism of signed graphs $\widehat{G} \to
\widehat{H}$.

\item \textit{At least one of the edges of $f(C)$ is in a subtree of type $T_4$ 
rooted at some $v_i, i \leq d-1$:}\\
In this case, all edges of $f(C)$ must be in subtrees of type $T_4$ rooted at
the same $v_i$. Assume again, without loss of generality, that $v_i$ is white
and the subtrees consist of the blue edges $v_ix_1, v_ix_2, \ldots, v_ix_m$,
with each $x_j$ black. Since $<_b, <_w$ is a special bipartite min ordering,
all vertices adjacent to $v_i$ by a bicoloured edge are smaller in $<_b$ than
$x_1, \ldots, x_m$. Therefore no such vertex can be in a list of a black
vertex in $C$. This again means that the image of $C$ is either included in
the set of edges $v_ix_1, v_ix_2, \ldots, v_ix_m$, or is disjoint from this
set of edges. We can delete all vertices $x_1,\ldots,x_m$ from the lists of
all black vertices of $C$ and repeat as above.

\item \textit{The edges of $f(C)$ are included in the set of edges on the
 path $v_dv_{d+1}\ldots v_k$ and in the subtrees of types $T_3$ or $T_4$ 
 rooted at $v_d, \ldots, v_k$:}\\
In this case, the vertices in the lists of the cycle $C$ are joined only by blue 
edges, and there is no homomorphism of signed graphs $\widehat{G} \to \widehat{H}$.
\end{itemize}

After we modified the image of one negative cycle $C$ of $\widehat{H}$, we
proceed to modify another, until we either obtain a homomorphism of signed graph,
or find that no such homomorphism exists. The algorithm is polynomial, because
arc consistency can be performed in linear time~\cite{fv}, and each modification
removes at least one vertex of $H$ from the list of at least one vertex of
$G$. Recall that the graph $H$ is fixed, and hence its number of vertices is a
constant $k$. If $G$ has $n$ vertices, then this step will be performed at most
$kn$ times.
\qed \end{proof}

\section{Reflexive trees} \label{sec:reflexive}

We now turn to reflexive trees, and hence in this section, $\widehat{H}$ will
always be a reflexive tree. We may have red, blue, or bicoloured loops, but we
may again assume that all non-loop unicoloured edges are of the same colour (blue
or red). 

\begin{lemma}\label{raining}
If $\widehat{H}$ contains one of the reflexive trees from the family $\cal G$ in
Figure~\ref{fig:forbgraphsref} as an induced subgraph, then
$\textsc{List-S-Hom}(\widehat{H})$ is NP-complete.
\end{lemma}

\begin{figure}
\centering
\includegraphics[scale=0.85]{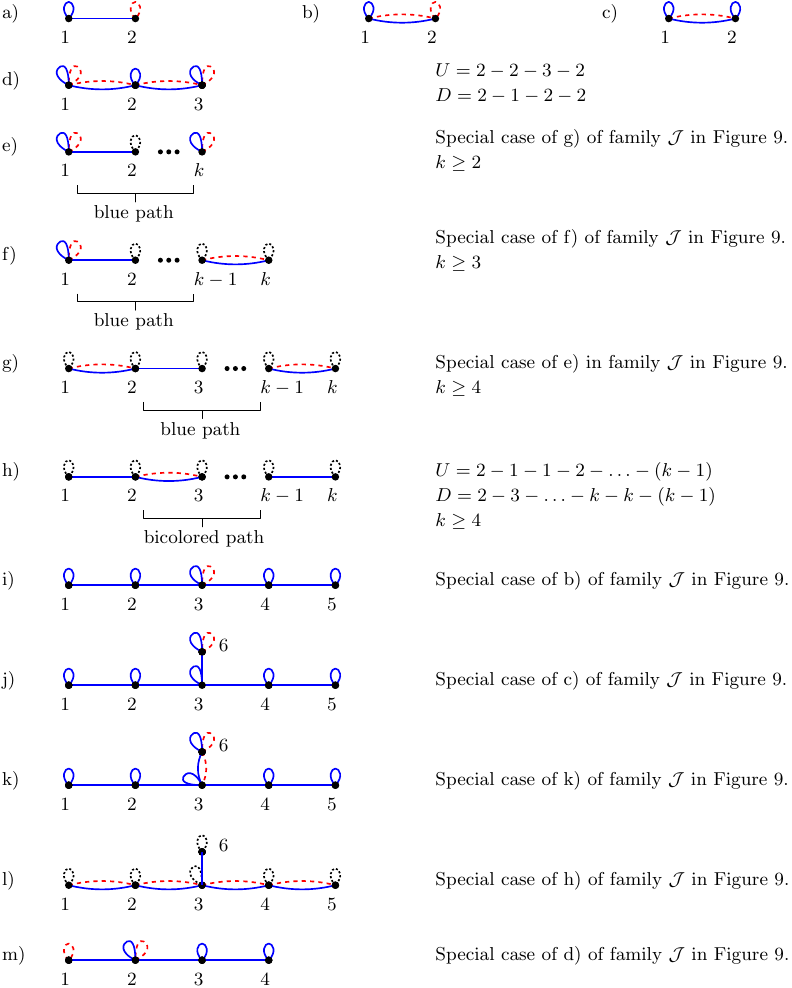}
\caption{The family $\cal G$ of signed reflexive trees with NP-complete problems. 
(The dotted loops can be either blue, red or bicoloured.)}
\label{fig:forbgraphsref}
\end{figure}

\begin{proof} 
The signed trees in a), b) and c) are themselves s-cores with more than two
edges, so it follows from Theorem~\ref{thm:dichotomy} that they yield NP-complete
problems. The signed trees in d), f), g), and h) have chains indicated in Figure~\ref{fig:forbgraphsref},
and hence also yield NP-complete problems by Theorem~\ref{thm:chain}. The remaining cases
are again handled in more general context in the next section, as indicated in Figure~\ref{fig:forbgraphsref}.

\qed \end{proof}

The next lemma is used to prove that in all polynomial cases $\widehat{H}$ is a caterpillar.  
Although this  section is restricted to reflexive graphs, we will prove it in greater generality 
for future use in a later section. To that end let $F_2$ be the graph in Figure~\ref{fig:bipclaw} 
where each loop on the three leaves may or may not be present.  Thus, $F_2$ represents 
a family of graphs, but we will abuse notation and simply refer to $F_2$ as any member 
of that family.

\begin{lemma} \label{lem:bipclaw}
If the underlying graph $H$ contains the graph $F_2$ in
Figure~\ref{fig:bipclaw}, then the problem $\textsc{List-S-Hom}(\widehat{H})$ is
NP-complete.
\end{lemma}

\begin{figure}
\centering
\includegraphics[scale=1]{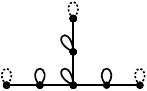}
\caption{The subgraph $F_2$.}
\label{fig:bipclaw}
\end{figure}

\begin{proof}
Deciding if there exists a list homomorphism (of an unsigned graph) to the graph
$F_2$ is NP-complete, as stated in~\cite{feder2003bi} (and proved using results 
in~\cite{feder1999list} and~\cite{feder2007bitr}). It would be natural to attempt a 
direct reduction of $\textsc{List-Hom}(F_2)$ to $\textsc{List-S-Hom}(\widehat{H})$,
as we have done here for the proof of Lemma~\ref{lem:tripleclaw}. However, this 
is complicated by the fact that the loops in $\widehat{H}$ can be red, blue, or 
bicoloured. Therefore, below we proceed on a different path, adapting to our setting 
the proof of the reflexive case from from~\cite{feder1998list} (see Theorem 2.3 
in that paper).

Suppose that $\widehat{F}_2$ is a subgraph of $\widehat{H}$ with underlying
graph $F_2$, and suppose that $\widehat{F}_2$ has been switched so that all
non-loop edges are at least blue. Label the leaves of $\widehat{F}_2$ by
$0, 1, 2$, and their respective neighbours by $0^+,1^+,2^+$, and finally label 
the central vertex by $c$.

If all the unicoloured loops in $\widehat{F}_2$ are blue, then we may restrict
the input to blue (there is no advantage to switching).  The NP-complete problem 
$\textsc{List-Hom}(F_2)$ \cite{feder2003bi} reduces to 
$\textsc{List-S-Hom}(\widehat{H})$.  

Similarly, if all unicoloured loops in $\widehat{F}_2$ are red, then we can switch all non-loop
edges to red and apply the same logic. Thus we assume that there are both blue and red 
unicoloured loops in $\widehat{F}_2$. 

We first prove that if some edge $ci^+$,  $i \in \{0, 1, 2 \}$ is not bicoloured, 
then $\textsc{List-S-Hom}(\widehat{F}_2)$ is NP-complete by showing that the 
copy of $\widehat{F}_2$ contains a  member of the family $\mathcal{G}$ or 
$\mathcal{J}$.  

Note, any path between a blue loop and a red loop must have a vertex with a
bicoloured loop; otherwise, $\widehat{F}_2$ contains a) or b) from family $\mathcal{G}$.
Thus at least one of $c, 0^+, 1^+, 2^+$ has a bicoloured loop.

Next, if none of the edges $ci^+$ are bicoloured, then we either have a copy of 
c) from family $\mathcal{J}$ when there is a bicoloured loop at some $i^+$ or a 
copy of b) when there is bicoloured loop at $c$.  If one of the $ci^+$ edges is 
bicoloured, then we have a copy of k) from family $\mathcal{J}$.
Finally if two of the edges are bicoloured, then we have a copy of h) from family 
$\mathcal{J}$.  (We note that the chain in h) is applicable even if the edges $12$ 
or $45$ are unicoloured.)  Thus, all edges $ci^+$ are bicoloured.

We now finish the proof using a modification of the proof in~\cite{feder1998list}.
Given distinct $i$ and $j$ in  $\{0,1,2\}$ and distinct subsets $I$ and $J$ of
$\{ 0, 1, 2 \}$, an \emph{$(i,I,j,J)$-chooser} is a path $\widehat{P}$ with
endpoints $a$ and $b$, together with a list assignment $L$, such that the
following statement holds. For each list homomorphism $f$ from $\widehat{P}$ 
to $\widehat{F}_2$, either $f(a) = i$ and $f(b) \in I$ or $f(a) = j$ and $f(b)
\in J$.  Moreover, for each $i' \in I$ and $j' \in J$, there are list
homomorphisms $g_1, g_2$ from $\widehat{P}$ to $\widehat{F}_2$ such that
$g_1(a) = i, g_1(b) = i'$ and $g_2(a) = j, g_2(b) = j'$.

Suppose $\widehat{P}$ is a $(0, \{0,1\}, 1, \{1,2\})$-chooser, $\widehat{P}'$
is $(0, \{1,2\}, 1, \{2,0\})$-chooser, and $\widehat{P}''$ is a $(0, \{2,0\},
1, \{0,1\})$-chooser.  Let $\widehat{T}$ be the tree obtained by identifying
the $b$ vertices in the three choosers and labelling the leaves respectively
as $a, a', a''$.  It is easy to verify that $\widehat{T}$ admits a
list-homomorphism to $\widehat{F}_2$ if, and only if, the triple $(a,a',a'')$
does not map to either $(0,0,0)$ or $(1,1,1)$. Consequently, we can reduce an
instance of {\sc Not-All-Equal SAT} to $\textsc{List-S-Hom}(\widehat{F}_2)$.
For each clause in the instance, create a copy of $\widehat{T}$ and identify
the vertices $(a, a', a'')$ with the three literals in the clause.

It remains to construct the choosers. First, we build a 
$(0, \{0, 2 \}, 1, \{ 1, 2 \})$-chooser.  By symmetry we then have 
$(i, \{ i, k \}, j, \{ j, k \})$-choosers for any distinct $i, j, k \in \{ 0, 1, 2 \}$.  Let $Q$ be a 
path on $q_0, q_1, \dots, q_{10}$ with lists 
$$
\begin{array}{lcl}
L(q_0) = \{ 0, 1 \}             & \hspace{1cm} & L(q_6) = \{ 0^+, 2^+, 1 \} \\
L(q_1) = \{ 0^+, 1^+ \}     & \hspace{1cm} & L(q_7) = \{ 0^+, c, 1^+ \} \\
L(q_2) = \{ 0, 1^+ \}         & \hspace{1cm} & L(q_8) = \{ 0, 2^+, 1 \} \\
L(q_3) = \{ 0^+, c, 1^+ \} & \hspace{1cm} & L(q_9) = \{ 0^+, 2^+, 1^+ \} \\
L(q_4) = \{ 0, 2^+, 1 \}     & \hspace{1cm} & L(q_{10}) = \{ 0, 2, 1 \} \\
L(q_5) = \{ 0^+, 2, 1^+ \} \\
\end{array}
$$

The path $\widehat{Q}$ has all edges blue.  In mapping $\widehat{Q}$ to $\widehat{F_2}$
first suppose $q_0$ maps to $0$.  Then $q_{10}$ either maps to $0$, in which case the
loop $0^+$ is traversed twice, or $q_{10}$ maps to $2$, in which case the loop at $0^+$
and the loop at $2^+$ are each traversed once.  In the both cases 
if the loop at $0^+$ is unicoloured red, then switch at $q_6$. In the latter case, 
if there is a red loop at $2^+$, then we switch at $q_8$. Note in the latter case the
bicoloured edges $0^+c$ and $c2^+$ allow the edges $q_6 q_7$ and $q_7 q_8$ 
to be of either colour. A similar reasoning shows 
$\widehat{Q}$ can map to $\widehat{F_2}$ with $q_0$ mapping to $1$ and 
$q_{10}$ mapping to either $1$ or $2$ but not to $0$.  Thus $\widehat{Q}$ is a 
$(0, \{ 0, 2 \}, 1, \{ 1, 2 \})$-chooser.

The $( 0, \{0\}, 1, \{2\} )$-chooser $\widehat{R}$ is a path with vertices 
$r_0, \dots, r_6$ and lists $$\{ 0, 1 \}, \{ 0^+, 1^+ \}, \{ 0, 1^+ \}, \{ 0^+, c \}, 
\{ 0, 2^+ \}, \{ 0^+, 2^+ \}, \{ 0, 2 \}.$$ All edges are blue.  When $\widehat{R}$
maps to the edge $00^+$, no switching is required as $00^+$ is at least blue.  
When $\widehat{R}$ maps to the path $1,1^+,1^+,c,2^+,2^+,2$, switching at 
$r_2$ (respectively $r_4$) is required when there is a unicoloured red loop at 
$1^+$ (respectively $2^+$).

The required choosers are defined as follows. First, $\widehat{P}$ is the $(0,
\{ 0 \}, 1, \{ 2 \})$-chooser followed by the $(0, \{ 0, 1 \}, 2, \{ 1, 2
\})$-chooser.  Next $\widehat{P}'$ is the concatenation of the $(0, \{0\}, 1,
\{ 2 \})$-chooser, the $(0, \{ 1 \}, 2, \{ 2 \})$-chooser, the $(1, \{ 1 \},
2, \{ 0 \})$-chooser, and the $(1, \{ 1, 2 \}, 0, \{ 0, 2 \})$-chooser. 
Finally $\widehat{P}''$ is the concatenation of the $(0, \{ 2 \}, 1, \{ 1
\})$-chooser and the $(2, \{0,2\}, 1, \{0,1\})$-chooser.
\qed \end{proof}

A tree $H$ is a {\em caterpillar} if it contains a path $P=v_1 \ldots v_k$ such 
that each vertex of $H$ is on $P$ or is adjacent to a vertex of $P$. Note
that the path $P$, which we again call the {\em spine} of $H$, is not unique,
and we sometimes make it explicit by saying that $H$ is a caterpillar {\em
with spine $P$}. A vertex $x$ not on $P$ is adjacent to a unique neighbour
$v_i$ on $P$, and we call the edge $v_ix$ (with the loop at $x$) {\em the subtree rooted at $v_i$}. A
vertex on the spine can have more than one subtree rooted at it. 
We say that a signed graph $\widehat{H}$ whose underlying graph $H$ is a reflexive caterpillar 
is a {\em good caterpillar with respect to the spine $v_1 \ldots v_k$} if the bicoloured edges of 
$\widehat{H}$ form a connected subgraph, the unicoloured non-loop edges all have the same 
colour $c$, and there exists an integer $d$, with $1 \leq d \leq k$, such that
\begin{itemize}
\item all edges on the path $v_1v_2 \ldots v_d$ are bicoloured, and all
edges on the path $v_dv_{d+1} \ldots v_k$ are unicoloured with colour $c$, 
\item all loops at the vertices $v_1, \ldots, v_{d-1}$ and all non-loop
edges of the subtrees rooted at these vertices are bicoloured,
\item all loops at the vertices $v_{d+1}, \ldots, v_k$ and all edges
and loops of the subtrees rooted at these vertices are 
unicoloured with colour $c$,
\item if $v_d$ has a bicoloured loop, then all children of $v_d$ 
with bicoloured loops are adjacent to $v_d$ by bicoloured edges,
\item if $v_d$ has a unicoloured loop of colour $c$, then all children 
of $v_d$ have unicoloured loops of colour $c$, and are adjacent 
to $v_d$ by unicoloured edges, and
\item if $d < k$, then the loops of all children of $v_d$ adjacent 
to $v_d$ by unicoloured edges also have colour $c$.
\end{itemize}
The vertex $v_d$ will again be called the {\em dividing vertex}. We also say
that $\widehat{H}$ is a good caterpillar {\em with preferred colour $c$}.
Figure~\ref{fig:example_good_ref} (on the left) shows an example of good caterpillar with
preferred colour blue. We emphasize that in the case $d=k$ 
(depicted in Figure~\ref{fig:example_good_ref} on the right), it
is possible (if $v_d$ has a bicoloured loop) that $v_d$ has some children with 
red loops and some with blue loops, adjacent to $v_d$ by unicoloured edges.

\begin{figure}[t]
\centering
\includegraphics[scale=1]{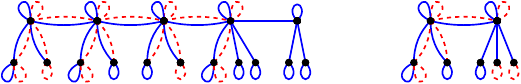}
\caption{Two good caterpillars with preferred colour blue: with $d < k$ (left), with $d=k$ (right).}
\label{fig:example_good_ref}
\end{figure}

Let ${\cal G}$ be the family of signed graphs depicted in
Figure~\ref{fig:forbgraphsref}, together with the family of complementary
signed graphs where all unicoloured edges and loops are red, rather than blue,
and vice versa. Note that the complementary signed graphs are not switching
equivalent to the original signed graphs because switching does not change the
colour of loops.

\begin{lemma} \label{lem:forbref}
Let $\widehat{H}$ be a reflexive signed tree. Then $\widehat{H}$ is a good 
caterpillar if and only if it does not contain any of the graphs in the 
family ${\cal G}$ as an induced subgraph, and the underlying graph $H$ 
does not contain the graph $F_2$.
\end{lemma}

\begin{proof}
It is easy to see that none of the signed reflexive trees in Figure~\ref{fig:forbgraphsref} 
is a good caterpillar. By symmetry, the same is true for their complementary signed graphs. 
It is also clear that the graph $F_2$ (from Figure~\ref{fig:bipclaw}) is not a caterpillar. 
We proceed to show that if the signed reflexive trees from family ${\cal G}$ in 
Figure~\ref{fig:forbgraphsref} are excluded as induced subgraphs, then $\widehat{H}$ 
is a good caterpillar with preferred colour blue. (The complementary exclusions 
produce a good caterpillar with preferred colour red.) 
Since the graphs g) are absent, the bicoloured non-loop edges induce a
connected subgraph. The exclusion of family h) similarly ensures
that all unicoloured non-loop edges induce a connected subgraph. 
By grouping all bicoloured non-loop edges before all unicoloured non-loop edges, 
we conclude that there exists a spine $P=v_1 \ldots v_k$, and a dividing vertex 
$v_d$.  Thus, all edges between $v_1, \dots, v_{d-1}$, and (since l) is excluded) 
all edges to their children, are bicoloured. The exclusion of b) and c) ensures 
each bicoloured non-loop edge has a bicoloured loop on (at least) one of its 
endpoints.  Forbidding the family d) ensures the vertices $v_1, \dots, v_{d-1}$ 
all have bicoloured loops.

The subgraph induced by $v_{d+1}, \dots, v_k$ and their children must contain only 
blue edges since the edge $v_{d}v_{d+1}$ is blue and the bicoloured edges induce 
a connected subgraph. Forbidding a), e), and f) implies all the loops in this subgraph 
are also blue. (Recall that when we say blue we always mean unicoloured blue.) 

Now we distinguish two cases. If $v_d$ has a blue loop then by excluding families 
a), b), c) and d) we conclude that all edges to its children are blue and all loops of 
its children are also blue. In the case $v_1 = v_d$, if there is a bicoloured loop on 
exactly one leaf of $v_1$ (respectively $v_k$), we renumber the vertices so that
this leaf becomes the first vertex of the spine, $v_1$. (If it was a leaf of $v_1$, this
involves a small shift of subscripts, if it was a leaf of $v_k$, it also involves a reversal
of the ordering of subscripts.) 

Now suppose that $v_d$ has a bicoloured loop. Excluding family e) ensures that any 
child of $v_d$ with a bicoloured loop must be adjacent to $v_d$ by a bicoloured edge. 

Finally if $d < k$, case m) implies that we can choose the spine so that no child of 
$v_{d}$ has a red loop.

Families i), j), and k) ensure when there is a single bicoloured loop or a
single bicoloured non-loop edge, the spine can be chosen to begin with 
this loop or edge.
\qed \end{proof}

\begin{theorem}
Let $\widehat{H}$ be a reflexive tree. If $\widehat{H}$ is a good caterpillar,
then the problem $\textsc{List-S-Hom}(\widehat{H})$ is polynomial-time
solvable. Otherwise, $H$ contains $F_2$ from Figure~\ref{fig:bipclaw}, or
$\widehat{H}$ contains one of the signed graphs in family ${\cal G}$ as an
induced subgraph, and the problem is NP-complete.
\end{theorem}

Suppose that $\widehat{H}$ is not a good caterpillar. If $H$ is not a
caterpillar, then it contains $F_2$ from Figure~\ref{fig:bipclaw}, and the
problem is NP-complete by Lemma~\ref{lem:bipclaw}. Otherwise, $\widehat{H}$ 
contains an induced subgraph from ${\mathcal G}$, and the problem is NP-complete 
by Lemma~\ref{raining}.

We prove the first statement. Thus assume that $\widehat{H}$ is a good
caterpillar, with spine $v_1 \ldots v_k$ and dividing vertex $v_d$. By
symmetry, we may assume it is a good caterpillar with preferred colour blue.
We distinguish three types of rooted subtrees.
\begin{itemize} 
\item Type $T_1$: a bicoloured edge $v_ix$ with a bicoloured loop on $x$;
\item Type $T_2$: a bicoloured edge $v_ix$ with a unicoloured loop on $x$;
\item Type $T_3$: a blue edge $v_ix$ with a unicoloured loop on $x$.
\end{itemize}

There is a general version of min ordering we can use in this context. 
A {\em min ordering} of a graph $H$ is a linear ordering
$<$ of the vertices of $H$, such that for vertices $x < x', y < y'$, if $xy',
x'y$ are both edges in $H$, then $xy$ is also an edge in $H$. It is again the
case that if a graph $H$ admits a min ordering, then the list
homomorphism problem for $H$ can be solved in polynomial time by arc
consistency followed by making the minimum choice in each list~\cite{fv}. Suppose
again that $\widehat{H}$ is a good caterpillar with spine $v_1 \ldots v_k$ and
preferred colour blue. A {\em
special min ordering} of $\widehat{H}$ is a min ordering of the underlying
graph $H$ such that for any vertices $v_i, x, x'$ with edges $v_ix, v_ix'$ we
have $x < x'$ if
\begin{itemize}
\item
the edge $v_ix$ is bicoloured and the edge $v_ix'$ 
is blue, or
\item
$x$ has a bicoloured loop and $x'$ a unicoloured loop, or 
\item
$x$ has a blue loop and $x'$ has a red loop.
\end{itemize}

\begin{lemma} \label{minoo}
Every good caterpillar $\widehat{H}$ admits a special min ordering.
\end{lemma}

\begin{proof}
It is again easy to see that the ordering $v_1 < v_2 < \ldots < v_k$ of
$V(\widehat{H})$ in which the children of each $v_i$ are ordered between $v_i$
and $v_{i+1}$ is a min ordering of the underlying graph $H$. We may again
assume that $\widehat{H}$ has preferred colour blue. To ensure that $<$ is a
special min ordering of $\widehat{H}$, we make sure that after each vertex
$v_i$ with $i=1, 2, \ldots, d-1$, we first list the leaves of subtrees of type
$T_1$, then the leaves of subtrees of type $T_2$ with blue loop, and last the
leaves of subtrees of type $T_2$ with red loop. If $d=k$, then we proceed the
same way also after $v_d$, and then we list the leaves of subtrees of type $T_3$
with blue loop, and last the leaves of subtrees of type $T_3$ with red loop.
If $d < k$, we list after $v_d$ first the leaves of subtrees of type $T_1$,
then the leaves of subtrees of type $T_2$ with blue loop, then the leaves of 
subtrees of type $T_2$ with red loop, and last the leaves of subtrees of
type $T_3$. For vertices $v_i, i > d$, there are only subtrees of type $T_3$,
and their leaves can be listed in any order.
\qed \end{proof}

We now describe our polynomial-time algorithm. 
As in the irreflexive case, we first perform the arc
consistency test to check for the existence of a homomorphism of the
underlying graphs ($G$ to $H$).
Then we also perform the bicoloured arc consistency test. If we obtain
an empty list, there is no list homomorphism. 
Otherwise, taking again the
minima of all lists (in the special min ordering $<$) defines a list homomorphism $f\colon G \to H$ of the
underlying graphs by~\cite{fv}, and again by bicoloured arc consistency test we have that
$f$ maps bicoloured edges of $\widehat{G}$ to bicoloured edges of
$\widehat{H}$. Therefore, by
Lemma~\ref{alter} and the remarks following it, $f$ is also a list
homomorphism of the signed graphs $\widehat{G} \to \widehat{H}$, unless a negative
cycle $C$ of unicoloured edges of $\widehat{G}$ maps to a positive closed walk
$f(C)$ of unicoloured edges in $\widehat{H}$, or a positive cycle $C$ of
unicoloured edges of $\widehat{G}$ maps to a negative closed walk $f(C)$ of
unicoloured edges in $\widehat{H}$. 
The minimum choices in all lists imply
that no vertex $x$ of $C$ can be mapped to an image $y$ with $y < f(x)$. 
We proceed to modify the images of such cycles $C$ one by one,
in the order of increasing smallest vertex in $f(C)$ (in the ordering $<$), until we 
either obtain a homomorphism of signed graphs, or we find that no such 
homomorphism exists.

Let $w$ be the leaf of the last subtree of type $T_2$ rooted at $v_d$ (we let $w = v_d$ if 
$v_d$ has no subtree of type $T_2$). We note that if $d < k$, then all edges and loops 
amongst the vertices that follow $w$ in $<$ are blue, by the properties of a special min 
ordering. Also note that since the edges of $f(C)$ are unicoloured, they do not include
a bicoloured loop on $v_d$ (if there is one). We distinguish three possible cases.

\begin{itemize}
\item \textit{At least one vertex $y$ of $f(C)$ satisfies $y \leq w$:}\\
The only unicoloured closed walks including $y$ are (red or blue) loops, so
$f$ maps the entire cycle $C$ to $y$. As in the reflexive case, we may remove
$y$ from all lists of vertices of $C$ and continue seeking a better
homomorphism of the underlying graphs ($G$ to $H$).
\item \textit{All vertices of $f(C)$ except for $v_d$ follow $w$ in the order $<$ and 
$d < k$, or $d=k$ and $v_d$ does not have a subtree of Type $T_3$ with red loop: }\\
In this case $C$ is a negative cycle of unicoloured edges. The subgraph of
$\widehat{H}$ induced by the vertices after $w$ (in the order $<$) has only blue
edges and loops. Thus there is no homomorphism of signed graphs mapping
$\widehat{G} \to \widehat{H}$.
\item \textit{All vertices of $f(C)$ except for $v_d$ follow $w$ in the order $<$, $d=k$ 
and $v_d$ has a subtree of Type $T_3$ with red loop:}\\
In this case a fairly complex situation may arise because $f(C)$ can be a
closed walk using both red and blue loops, along with blue edges; see below.
\end{itemize}

We now consider the final case in detail. Since $f$ chooses minimum possible 
values of images (under $<$), we could only modify $f$ by mapping some vertices 
of $C$ that were taken by $f$ to a vertex with a blue loop, to vertex with a red loop 
instead, if lists allow it. We show how to reduce this problem to solving a system 
of linear equations modulo two, which can then be solved in polynomial time by 
(say) Gaussian elimination.
We begin by considering the pre-image (under $f$) of all vertices in the
subtrees of type $T_3$ rooted at $v_d$.  We denote by $P$ the set of vertices
$v \in V(G)$ with $f(v)$ equal to a vertex with a blue loop and by $N$ the set of
vertices $v \in V(G)$ with $f(v)$ equal to a vertex with a red loop. We say that
a vertex $x$ of $G$ is a \emph{boundary point} if $f(x)=v_d$. The set of
boundary points is denoted by $B$. Thus the pre-image of the subtrees of type
$T_3$ rooted at $v_d$ is the disjoint union $B \cup P \cup N$.  We now focus
on the subgraph $\widehat{G}'$ of $\widehat{G}$ induced by $B \cup P \cup N$. 
A \emph{region} is a connected component of $\widehat{G}' \setminus B$
together with all its boundary points, i.e. between any pair of vertices in a
region there is a path with no boundary point as an internal vertex.

Given a region $r$ and boundary points $x$ and $y$ (not necessarily distinct),
we construct (possibly several) boolean equations on the corresponding
variables, using the same symbols $x, y,$ and $r$. The variables $x, y$
indicate whether or not the corresponding boundary vertices $x$ and $y$ should
be switched before mapping them with $f$ (true corresponds to switching), and
the variable $r$ indicates whether the region $r$ will be mapped by $f$ to a
blue loop or a red loop (true corresponds to a blue loop). The equations
depend of the parity and the sign of walks between the two vertices. If $c$
and $d$ denote parities (even or odd), we say a walk $W$ from $x$ to $y$ in
$\widehat{G}'$ is a \emph{$(c,d)$-walk} if it contains no boundary points
other than $x$ and $y$, the parity of the number of blue edges in $W$ is $c$,
and the parity of the number of red edges in $W$ is $d$. The equations
generated by the $(c,d)$-walks are as follows.
\begin{itemize}
\item \textit{(odd,odd)-walk:} We add the equation $x = y + 1$.
This ensures that exactly one of the boundary vertices has to be switched, in
particular $x$ and $y$ must be distinct. The image of the walk must be
uni-balanced or anti-uni-balanced (as the whole walk maps to exactly one subtree of
type $T_3$).  A walk with an even number of edges but an odd number of red edges 
is neither. However, if we switch at exactly one of the endpoints, we can freely 
map all of the non-boundary points to a blue loop or a red loop. 
\item \textit{(even,even)-walk:} We add the equation $x = y$. 
The reasoning is similar to the previous case.
\item \textit{(odd,even)-walk:} We add the equation $x = y + r + 1$. 
The image of the walk is a closed walk with an odd number of edges and positive sign. 
Thus if both or neither of $x$ and $y$ are switched, then the walk remains
positive and $r=1$.  Conversely, switching exactly one of $x$ or $y$ 
makes the walk negative, and $r=0$.
\item \textit{(even,odd)-walk:} We add the equation $x = y + r$. 
The argument is analogous to the previous case.
\end{itemize}

It is possible that there are several kinds of walks between the same $x, y$, 
but we only need to list one of each kind, so the number of equations is 
polynomial in the size of $G$. 
A simple labelling procedure can be used for determining which kinds of walks 
exist, for given boundary points $x$ and $y$ and a region $r$. We start at the
vertex $x$, and label its neighbours $n_x$ by the appropriate pairs $(c,d)$,
determined by the signs of the edges $xn_x$. Once a vertex is labelled
by a pair $(c,d)$, we correspondingly label its neighbours; a vertex is only given
a label $(c,d)$ once even if it is reached with that label several times. Thus a
vertex has at most four labels. Any time a vertex receives a new label its neighbours
are checked again. The process ends in polynomial time (in the size of the region)
as each edge of the region is traversed at most four times. The result is inherent 
in the labels obtained by $y$.

Finally, for each region we examine the connected component of the
non-boundary vertices.  Since the arc consistency procedure was done in the
first step of the algorithm, all lists of non-boundary points for a given
region are the same. Also, by the ordering $<$, these lists must only contain
leaves of $v_d$.  Thus, the non-boundary vertices of the region must map to a
single loop.  We ensure the choice of the loop is consistent with the lists of
each region.  If the lists of vertices of some region do not contain a vertex
with a red loop, then we add the equation $r=1$ for the region.  Similarly, if
the lists do not contain a blue loop, then we add the equation $r=0$.  

Such a system of boolean linear equations can be solved in polynomial time.
Also, the system itself is of polynomial size measured by the size of $\widehat{G}$. 
This completes the proof.

\section{General trees}

In this section we handle signed trees $\widehat{H}$ in general, i.e., trees in 
which some vertices have loops while others do not. 
In homomorphism problems, reflexive and irreflexive bipartite target graphs $H$ 
tend to share some similarities, cf. e.g. \cite{rets,feder1998list,feder1999list}, and 
also both tend to be simpler. For instance, the general version of list homomorphisms 
for graphs with possible loops~\cite{feder2003bi} is significantly more involved than
both the reflexive and irreflexive bipartite cases~\cite{feder1998list,feder1999list}.
Similarly, considering general signed trees with possible loops introduces an 
additional level of difficulty.

To simplify the descriptions, we assume, without loss of generality, that 
\emph{all non-loop unicoloured edges are blue,} unless noted otherwise.
In Figure~\ref{fig:mixed_merged} we introduce our main NP-complete cases.

\begin{figure}[!htbp]
\centering
\includegraphics[height=0.97\textheight]{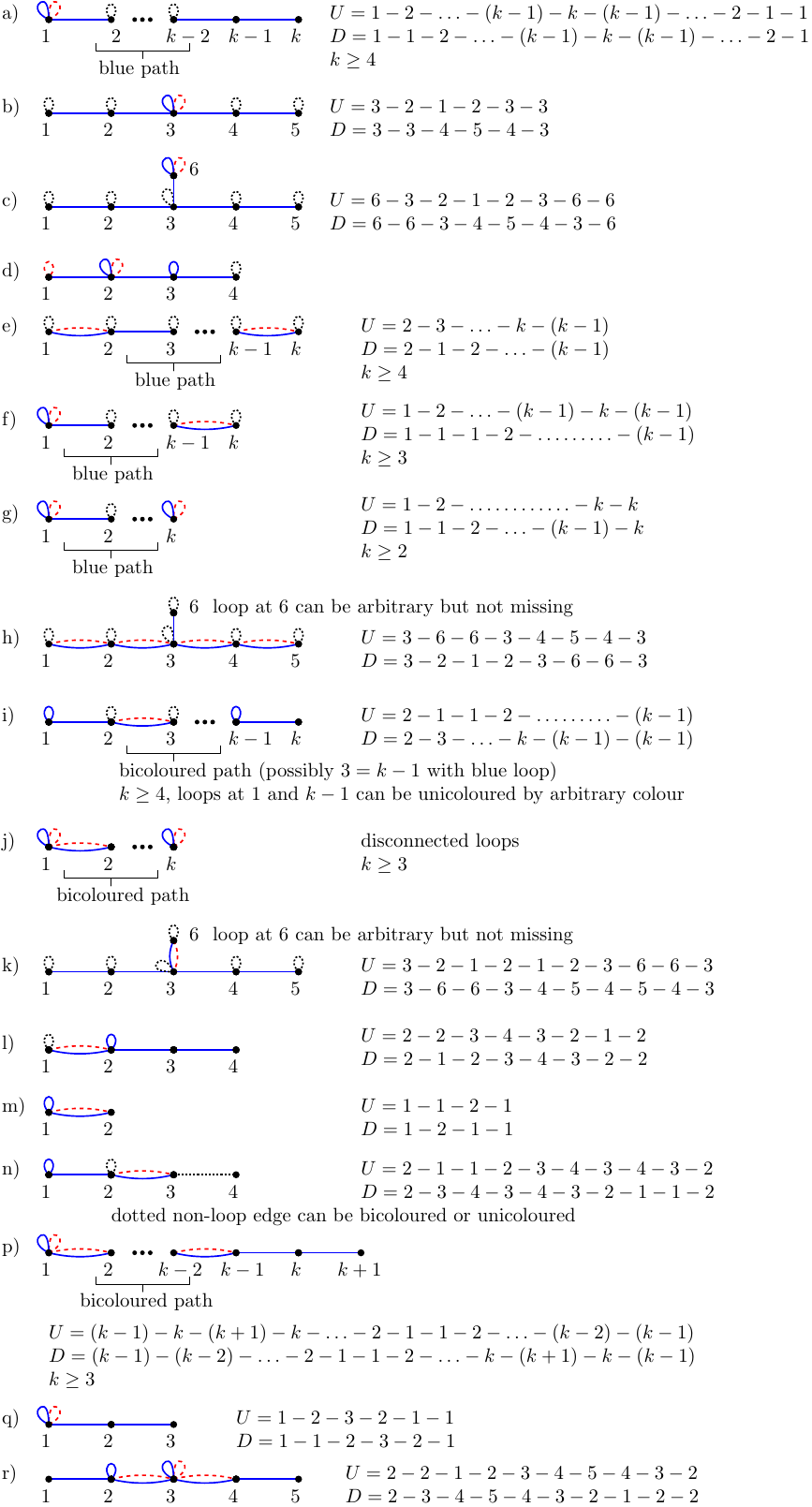}
\caption{The family ${\cal J}$. (The dotted loops can be arbitrary or missing, unless stated otherwise.)}
\label{fig:mixed_merged}
\end{figure}

We first focus on signed trees $\widehat{H}$ without bicoloured non-loop edges. If there are no bicoloured loops either, then
Theorem~\ref{thm:bordeaux} implies that $\textsc{List-S-Hom}(\widehat{H})$ is NP-complete when $\widehat{H}$ has both
a red loop and a blue loop, or when the underlying graph is not a bi-arc tree. We now introduce NP-complete
cases when bicoloured loops are allowed.

\begin{lemma} \label{lem:mixedup}
If $\widehat{H}$ contains any of the graphs a)-d) in the family ${\cal J}$ in Figure~\ref{fig:mixed_merged}, then the problem 
$\textsc{List-S-Hom}(\widehat{H})$ is NP-complete.
\end{lemma}

\begin{proof}
For each of the signed graphs a), b), and c) in family $\cal J$, we can apply Theorem~\ref{thm:chain}. 
The figure lists a chain for each of these forbidden subgraphs.

In the final case d), we reduce {\sc Not-All-Equal SAT} to
$\textsc{List-S-Hom}(H,\pi)$ where $(H,\pi)$ is the signed graph d) in
family ${\cal J}$. Let $(T',\sigma')$ be the signed graph
with the list assignments and signature shown in Figure~\ref{fig:gadget_d}. For
each clause $(x,y,z)$ in the instance of {\sc Not-All-Equal SAT}, we create a
copy of $(T',\sigma')$ identifying the leaves $x,y,z$ in $T'$ with the variables
in the clause.

\begin{figure}
\centering
\includegraphics[scale=1]{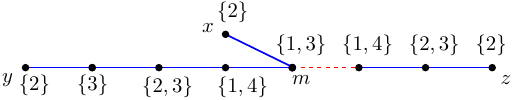}
\caption{The gadget $(T',\sigma')$ for the case d) in family ${\cal J}$.}
\label{fig:gadget_d}
\end{figure}

We claim that $(T',\sigma')$ has a list homomorphism to $(H,\pi)$ if and only
if we switch at exactly one or two elements of $\{ x, y, z \}$. We can then
view the switching at one of $\{x,y,z\}$ as setting the variable to true and,
conversely, no switching as setting to false.

Consider a mapping of $(T',\sigma')$ to $(H,\pi)$. It is easy to see that
either both $x$ and $m$ are switched or neither is switched. We also observe
that if $m$ maps to $1$, then exactly one of $m$ or $y$ must be switched.  On
the other hand, if $m$ maps to $3$, then neither or both of $m$ and $y$ is
switched.  (In the first case the image of the $(m,y)$-path is a negative walk, while
in the second case it is a positive walk.)  Thus, when $m$ maps to $1$,
exactly one of $x$ or $y$ is switched, and when $m$ maps to $3$, either both
or neither $x$ and $y$ is switched.  Finally, if $m$ maps to $1$, then we are
free to switch or not switch at $z$. On the other hand, if $m$ maps to $3$,
then we must switch at $z$ if and only if we do not switch at $m$. In conclusion,
with $m$ mapping to $1$ the following truth values are possible for $x, y, z$
respectively: $1, 0, 0; 1, 0, 1; 0, 1, 0; 0, 1, 1$, and with $m$ mapping to $3$
we obtain the possible triples $1, 1, 0$ and $0, 0, 1$ for the variables $x, y, z$.
These are precisely the not-all-equal values as claimed.
\qed \end{proof}

If bicoloured edges are present, we use the following result.

\begin{lemma} \label{lem:mixed}
If $\widehat{H}$ contains any of the graphs e)-n) in family ${\cal J}$ in Figure~\ref{fig:mixed_merged}, then the problem 
$\textsc{List-S-Hom}(\widehat{H})$ is NP-complete.
\end{lemma}

\begin{proof}
For each of the signed graphs e)-n) in family ${\cal J}$, except for the case j), we can apply Theorem~\ref{thm:chain}. The figure lists a 
chain for each of these forbidden subgraphs. The case j) follows from a result in~\cite{krokhin} implying that the 
problem is NP-complete if the vertices with loops of any colour are disconnected. Thus any signed graph $\widehat{H}$ that 
contains one of the signed graphs in the cases e)-n) of the family $\cal J$ as an induced subgraph has the problem 
$\textsc{List-S-Hom}(\widehat{H})$ NP-complete.
\qed \end{proof}

In cases p), q), and r) in Figure~\ref{fig:mixed_merged} we present three additional NP-complete trees we will use. Note 
that the case p) is an extension of case q), and the chains are also related. (Note that p) is also similar to a) in family $\cal J$.)
We note that the absence of a loop at $2$ is crucial for the chain in the case p). We also note, for the case r), that the absence 
of a loop at $4$ is crucial, while the edges $12$ or $45$ could be blue or bicoloured and the given chain would still apply.

Thus we have the following lemma.

\begin{lemma}
If $\widehat{H}$ contains any of the graphs p), q), r) in family ${\cal J}$ in Figure~\ref{fig:mixed_merged}, then the problem 
$\textsc{List-S-Hom}(\widehat{H})$ is NP-complete.
\end{lemma}

If $\widehat{H}$ is a signed graph, the {\em bicoloured part} of $\widehat{H}$ is the graph $D_{\widehat{H}}$
(with possible loops) consisting of all those edges and loops that occur as bicoloured edges and loops in $\widehat{H}$, 
and all the vertices they contain. (Thus vertices of $\widehat{H}$ not incident with a bicoloured edge or loop are deleted.)
Similarly, the {\em blue part} of $\widehat{H}$ is the graph $B_{\widehat{H}}$ with possible loops consisting of all those edges 
(and loops) that are at least blue in $\widehat{H}$. Since we assume all non-loop edges of $\widehat{H}$ are blue, every 
vertex of $\widehat{H}$ is included in $B_{\widehat{H}}$. (We may think of {\em B} as standing for "blue" and {\em D} as 
standing for ``double'', in the sense of having both colours.) 

We now denote by $\cal T$ the union of all the NP-complete tree families $\cal F, \cal G, \cal J$.
There are further cases that cause the problem to be NP-complete. Theorem~\ref{thm:bordeaux} implies, in the context of trees,
that the problem is NP-complete if there are no bicoloured edges or loops and there is both a red loop and a blue loop.
Any signed graph $\widehat{H}$ which is not irreflexive and has a bicoloured edge but no bicoloured loops yields an NP-complete 
homomorphism (and hence list homomorphism) problem by Theorem~\ref{thm:dichotomy}, since the s-core contains at least one 
unicoloured loop and one bicoloured edge (counted as two edges). As discussed earlier, if the vertices with loops of any fixed 
colour induce a disconnected graph, the problem is NP-complete by~\cite{krokhin}. Finally, as mentioned earlier, if the bicoloured 
part $D_{\widehat{H}}$ yields an NP-complete list homomorphism problem, then so does $\widehat{H}$, since for bicoloured 
inputs, this is the only part of $\widehat{H}$ that can be used. Thus $\textsc{List-S-Hom}(\widehat{H})$ is also NP-complete if 
the unsigned graph $D_{\widehat{H}}$ is not a bi-arc tree, i.e., contains one of the trees in Figures 3 and 4 of~\cite{feder2003bi}.
Moreover, if $\widehat{H}$ contains no red loops, then it is also true that if the blue part $B_{\widehat{H}}$ yields an NP-complete 
list homomorphism problem, then so does $\widehat{H}$. Indeed, if there are no red loops (or edges) in $\widehat{H}$, then for an
input signed graph $\widehat{G}$ that has only blue edges, there is no cause for switching. In other words a blue input $\widehat{G}$ 
admits a signed list homomorphism to $\widehat{H}$ if and only if $G$ admits a list homomorphism to $H$. This 
is a reduction from the list homomorphism problem for $B_{\widehat{H}}$ to the signed list homomorphism problem for $\widehat{H}$. 

We say that a signed tree is \emph{colour-connected} if each of the following subgraphs is connected: the subgraph spanned by non-loop
edges that are at least blue, the subgraph spanned by non-loop edges that are at least red, the subgraph spanned by non-loop edges that 
are bicoloured, the subgraph induced by the vertices with loops that are at least blue, the subgraph induced by the vertices with loops that 
are at least red, and the subgraph induced by the vertices with loops that are bicoloured.

We call a signed tree $\widehat{H}$ a \emph{good signed tree} if it satisfies the following conditions.

\begin{enumerate}
\item
If $\widehat{H}$ has no bicoloured edge, then all the loops are of the same colour (red or blue).
\item
If $\widehat{H}$ has a bicoloured non-loop edge, then it also has a bicoloured loop, or it has no loops at all.
\item
$\widehat{H}$ is colour-connected.
\item
The blue part $B_{\widehat{H}}$ is a bi-arc tree.
\item
$\widehat{H}$ contains no signed tree from the family $\cal T$.
\end{enumerate}

\subsection{Assuming no red loops}

In this subsection, we assume that $\widehat{H}$ has no red loops. It follows from the previous section, that if such
$\widehat{H}$ is not good, then $\textsc{List-S-Hom}(\widehat{H})$ is NP-complete. In particular, $\widehat{H}$ is 
colour-connected, since (as observed before), \cite{krokhin} implies that the problem is NP-complete if the vertices 
with loops of any colour are disconnected, and the family e) in $\cal J$ implies that the problem is NP-complete if
the subgraph spanned by non-loop edges that are bicoloured is not connected. Also recall that all unicoloured 
non-loop edges are assumed to be blue, and thus all non-loop edges that are at least red are in fact bicoloured. 
In the next subsection, we prove this fact (that signed trees that are not good have NP-complete problems) is true 
if we allow red loops as well. 

We first analyze the structure of good signed trees without red loops.

Let $\widehat{H}$ be a good signed tree with no red loops and at least one bicoloured loop. 
Since the blue part $B_{\widehat{H}}$ is a bi-arc tree, we can use the results of~\cite{feder2003bi}
and~\cite{feder2007bitr}, which together characterize bi-arc trees as trees in which vertices with loops
induce a connected subgraph, and which are either obtained from a reflexive caterpillar by deleting the
loops at a (possibly empty) subset of leaves (illustrated in Figure~\ref{fig:biarc_fig5}, repeated from
Figure 5 of~\cite{feder2003bi}), or obtained from an irreflexive $2$-caterpillar in one of the following ways:
(1) (possibly) adding a loop at a good vertex $v$, or (2) adding a loop at a good vertex $v$ and on one 
neighbour $w$ of $v$ which has the property that each neighbour of $w$ other than $v$ is a leaf, or (3) 
adding a loop at a good vertex $v$ and on a (possibly empty) set of neighbours of $v$ that are leaves. 
Here a {\em good vertex} is a vertex $v$ for which there does not exist a path $P$ with seven vertices, 
with the middle vertex $u$ connected to $v$ by a path (possibly with zero edges) which is disjoint from $P$. 
It is easy to see
that if $v$ is a good vertex, then there exists a spine in which $v$ is the first vertex, $v=v_1$ (and, in case
(2), the vertex $w$ is a child of $v$, not on the spine; similarly in case (3) the leaves of $v$ to which loops 
have been added are children of $v$ not on the spine). These cases are illustrated in Figure~\ref{fig:biarc_fig6}, 
repeated here from Figure 6 in~\cite{feder2003bi}. The two $2$-caterpillars in that figure will be called Type (a) 
and Type (b), as shown.

\begin{proposition}\label{precedens}
Let $\widehat{H}$ be a good signed tree without red loops but with at least one bicoloured loop.

Then $\widehat{H}$ is either 
\begin{itemize}
\item
obtained from a good reflexive caterpillar (with spine $v_1, v_2, \ldots, v_k$) by
\begin{itemize}
\item
removing loops at a subset $S$ of leaves, and
\item
optionally replacing any bicoloured edges $v_iu$ by blue edges for these leaves $u \in S$, or
\end{itemize}
\item
is a signed $2$-caterpillar (with spine $v_1, v_2, \ldots, v_k$) obtained from a bi-arc tree by
\begin{itemize}
\item
replacing each edge and loop by a bicoloured edge and loop (respectively),
\item
optionally, for Type (b) $2$-caterpillars, adding a blue loop at a leaf adjacent to $v_1$, and
\item
optionally adding, at a spine vertex $v_i$ or at a loopless child of a $v_i$, a blue edge leading to a new (loopless) leaf.
\end{itemize}
\end{itemize}
\end{proposition}

\begin{proof}

\begin{figure}[h]
\centering
\includegraphics[scale=1]{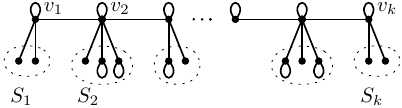}
\caption{Bi-arc caterpillars from~\cite{feder2003bi}.}
\label{fig:biarc_fig5}
\end{figure}

\begin{figure}[h]
\centering
\includegraphics[scale=1]{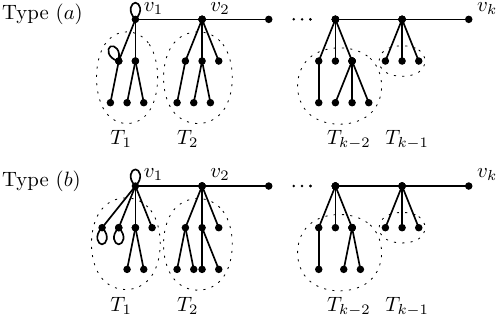}
\caption{Bi-arc $2$-caterpillars from~\cite{feder2003bi}.}
\label{fig:biarc_fig6}
\end{figure}

Assume first that the blue part $B_{\widehat{H}}$ is a bi-arc tree of the first type, in other words, a caterpillar 
with loops on all vertices on the spine and possibly some leaves. We now proceed analogously to the proof 
of Lemma~\ref{lem:forbref}, using the absence of signed trees from the family $\cal J$ instead of those from
the family $\cal G$. We sketch the analogy, and leave the detailed proof to the reader. For this, it helps to 
refer to the annotations in Figure \ref{fig:forbgraphsref} relating the cases of the family $\cal G$ to the more 
general trees in the family $\cal J$. It is also helpful to point out that the case h) of the family $\cal G$ is closely
related to the case i) of the family $\cal J$ (as well as to b) of the family $\cal F$). There is a common generalization
to all three, but it has a technical formulation we chose to omit, because other cases of the family $\cal J$ cover the
same situations; in particular the reader should note the case n), which is also helpful in the omitted proof. The 
principal difference from the proof in the reflexive case is caused by the requirement that certain loops in cases h) 
and l) in $\cal G$ have to remain present in the corresponding cases in family $\cal J$. This results in the fact that 
some vertices $v_1, v_2, \dots, v_{d-1}$ can have incident blue edges off the spine, as long as they lead to vertices 
without  loops, as enforced by the absence of the signed trees from the family $\cal J$.

Thus $\widehat{H}$ is indeed a caterpillar obtained from a reflexive signed caterpillar by removing 
loops at some leaves, and optionally replacing the bicoloured edges by blue edges to some of those
leaves.

In the remaining cases, the blue part $B_{\widehat{H}}$ is a $2$-caterpillar obtained by adding suitable
loops to an irreflexive bi-arc tree, cf. the two bi-arc trees in Figure~\ref{fig:biarc_fig6}. Specifically, there are two 
cases to consider.

In the first case, the blue part $B_{\widehat{H}}$ has two loops, at least one of which is bicoloured. According 
to~\cite{feder2007bitr}, we may choose the spine so that one loop of $B_{\widehat{H}}$ is at $v_1$ and the other at 
its child $u$. If the loop at $v_1$ is bicoloured in $\widehat{H}$, then the absence of p) and q) in family $\cal J$ 
implies that we may assume the spine consists of bicoloured edges 
only, and if a vertex $v_i$ on the spine has a (necessarily loopless) neighbour $w$ that is not a leaf, then the edge 
$v_iw$ is bicoloured. The neighbour $u$ has a loop and needs to be considered separately. We first claim that the edge 
$uv_1$ must be bicoloured, else $\widehat{H}$ contains the subtree n) from the family $\cal J$ (with $2$ corresponding 
to $v_1$), or g) from the family $\cal J$ (with $k=2$). Moreover, if the loop at $u$ is unicoloured, then $u$ must be a leaf, 
otherwise $\widehat{H}$ would contain r) from the family $\cal J$. If the loop at $v_1$ is unicoloured, then the loop at $u$
must be bicoloured (we assumed that a bicoloured loop exists). Now, unless $\widehat{H}$ arose from a reflexive 
caterpillar, it must contain a) from the family $\cal J$ (if $uv_1$ is blue), or l) from family $\cal J$ (if $uv_1$ is bicoloured). 
(Note that in both cases, the chain applies even if the edges $23, 34$ are bicoloured.) In conclusion, in this case we either 
have both loops at $v_1$ and $u$ (as well as the edge joining them) bicoloured, or the loop at $v_1$ and the edge 
$uv_1$ is bicoloured, the loop at $u$ is blue and a leaf. The former situation is depicted on the left of 
Figure~\ref{fig:mixed_from_irref} ($u$ is depicted as the child of $v_1$ in $T_1$), and the latter situation is a special case of the tree on the right, 
with only one child ($u$) of $v_1$ having a (blue) loop. Thus going from $D_{\widehat{H}}$ to $\widehat{H}$ we only 
added a blue loop on a leaf $u$ adjacent to $v_1$, and then added some blue edges leading to leaves from any spine
vertex $v_i$, or from any child of $v_2, v_3, \dots, v_k$, or from any child of $v_1$ other than $u$.

In the second case, the blue part $B_{\widehat{H}}$ has one loop at $v_1$ and possibly several other loops at leaf children 
of $v_1$. If the loop at $v_1$ is bicoloured, and possibly some of the loops at its children are also bicoloured, then the
proof proceeds exactly as in the previous case, concluding that any edge joining two vertices with loops must be bicoloured 
(else there would be a copy of the subtree n) from the family $\cal J$) and the children of $v_1$ with loops are leaves. If, 
say, leaf $u$ has a bicoloured loop and all other loops, including the loop at $v_1$, are blue in $\widehat{H}$, we again 
obtain a contradiction to the absence of a) from family $\cal J$ or l) from family $\cal J$, unless $\widehat{H}$ arose from 
a reflexive caterpillar. In conclusion, in this case, going from $D_{\widehat{H}}$ to $\widehat{H}$ involved only the addition 
of blue loops on leaves adjacent to $v_1$, and a possible addition of some blue edges from spine vertices or from non-loop 
children of spine vertices, leading to leaves as described.
\qed \end{proof}

\subsection{Allowing red loops}

We now consider signed graphs $\widehat{H}$ in which red loops are allowed. We denote by $\widehat{H'}$ 
the signed tree obtained from $\widehat{H}$ by deleting all vertices with red loops. We focus on the blue part $B_{\widehat{H'}}$ 
instead of $B_{\widehat{H}}$ because $\widehat{H'}$ has no red loops and satisfies the assumptions of Proposition~\ref{precedens}.

\begin{figure}[h]
\centering
\includegraphics[scale=1]{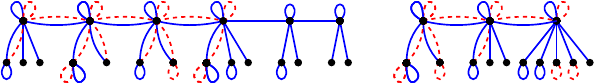}
\caption{Good signed trees obtained from a good reflexive tree by deleting loops.}
\label{fig:mixed_from_ref}
\end{figure}

\begin{figure}[h]
\centering
\includegraphics[scale=1]{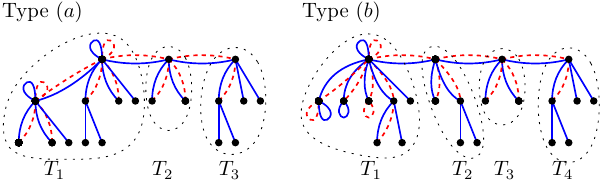}
\caption{Good signed trees obtained from a bi-arc tree as described in Proposition~\ref{description} .}
\label{fig:mixed_from_irref}
\end{figure}

\begin{proposition}\label{description}
Let $\widehat{H}$ be a good signed tree with at least one bicoloured loop.

Then $\widehat{H}$ is either 
\begin{itemize}
\item
obtained from a good reflexive caterpillar (with spine $v_1, v_2, \dots, v_k$) by
\begin{itemize}
\item
removing loops at a subset $S$ of leaves, and
\item
optionally replacing any bicoloured edges $v_iu$ by blue edges for these leaves $u \in S$, or
\end{itemize}
\item
is a signed $2$-caterpillar (with spine $v_1, v_2, \dots, v_k$) obtained from a bi-arc tree by
\begin{itemize}
\item
replacing each edge and loop by bicoloured edge and loop (respectively),
\item
optionally, for Type (b) $2$-caterpillars, adding a unicoloured loop at any leaf adjacent to $v_1$, and
\item
optionally adding, at a spine vertex $v_i$ or at a loopless child of a $v_i$, a blue edge leading to a new (loopless) leaf.
\end{itemize}
\end{itemize}
\end{proposition}

\begin{proof}
Since $\widehat{H'}$ (defined above) satisfies the assumptions of Proposition~\ref{precedens}, the tree $\widehat{H'}$ 
is described by the proposition, and we now consider where can the vertices of $B_{\widehat{H}} - B_{\widehat{H'}}$ 
be added, without violating any of the assumptions on $\widehat{H}$. Since the vertices with red loops must form 
a connected subgraph, they must be adjacent to each other and then to vertices with bicoloured loops. We now 
take in turn each case in the previous proof.

Consider the first case, when $\widehat{H'}$ is a caterpillar with reflexive spine vertices, and suppose $x$ is a
vertex in $\widehat{H}$ - $\widehat{H'}$ with a red loop. If $x$ is adjacent to a vertex on the spine of 
$\widehat{H'}$, then $\widehat{H}$ is indeed a good reflexive caterpillar with some loops on leaves removed.
If $x$ is adjacent to a leaf of $\widehat{H'}$ with a red loop, then either $\widehat{H}$ contains a copy of $F_2$, 
or $\widehat{H}$ is another good caterpillar with a different spine, and possibly different preferred colour, from which some loops at leaves have been removed.

In the second case, let $x$ be again a vertex of $\widehat{H}$ - $\widehat{H'}$ with a red loop. We claim it is adjacent to $v_1$ by a bicoloured edge; indeed it cannot be adjacent to a child $u$ of $v_1$, then $u$ would have to have a bicoloured loop or a red loop.  In this case, $\widehat{H}$ is either is a
caterpillar with reflexive spine and we are in the previous case, or we would have, in red, a path with three loops 
followed by two non-loops, which is NP-complete according to Theorem 5.1 of \cite{feder2003bi}, see Figure 3 in that paper.
\qed \end{proof}

In both cases of the proof above, we note that when the red loops are deleted (without deleting their vertices), we 
obtain a signed graph which also satisfies the assumptions of Proposition~\ref{precedens}. Moreover, if the red
loops are all changed to be blue, the same conclusion holds. These observations justify the following corollary.

\begin{corollary}\label{monitor}
Suppose $\widehat{H}$ is a signed tree.
If the blue part $B_{\widehat{H}}$ is not a bi-arc tree, then $\textsc{List-S-Hom}(\widehat{H})$ is NP-complete.
If the underlying unsigned tree is not a bi-arc tree, then $\textsc{List-S-Hom}(\widehat{H})$ is NP-complete.
\end{corollary}

It follows from the first statement of Corollary~\ref{monitor} that if a signed tree is not good then 
$\textsc{List-S-Hom}(\widehat{H})$ is NP-complete even if there are red loops in $\widehat{H}$.

We now state our main theorem of this section.

\begin{theorem}\label{provable}
If $\widehat{H}$ is a good signed tree, then $\textsc{List-S-Hom}(\widehat{H})$ is polynomial-time solvable. 
\end{theorem}

We can explicitly state the dichotomy classification as follows.

\begin{corollary}\label{mainko}
Let $\widehat{H}$ be a signed tree.

If any of the following conditions apply, then $\textsc{List-S-Hom}(\widehat{H})$ is NP-complete.

\begin{enumerate}
\item
$\widehat{H}$ has no bicoloured loop, but there is a bicoloured (non-loop) edge and a unicoloured loop.
\item
$\widehat{H}$ has no bicoloured edge, but there is a red loop and a blue loop.
\item
The bicoloured part $D_{\widehat{H}}$ is not a bi-arc tree, i.e., contains a subgraph from Figures 3 or 4 of~\cite{feder2003bi}.
\item
The blue part $B_{\widehat{H}}$ is not a bi-arc tree, i.e., contains a subgraph from Figures 3 or 4 of~\cite{feder2003bi}.
\item
$\widehat{H}$ contains a signed tree from the family $\cal T$.
\item
The set of vertices of $\widehat{H}$ with red (respectively blue, or at least blue, or bicoloured) loops induces a disconnected graph.
\end{enumerate}

If none of the conditions apply, then $\textsc{List-S-Hom}(\widehat{H})$ polynomial-time solvable.
\end{corollary}

We also state the result in the more usual complementary way, where the polynomial cases are enumerated first. Note that
here all the conditions are required to be satisfied to yield a polynomial case.

\begin{corollary}\label{malino}
Let $\widehat{H}$ be a signed tree. 
If all of the following conditions apply, then $\textsc{List-S-Hom}(\widehat{H})$ is polynomial-time solvable.

\begin{enumerate}
\item
If $\widehat{H}$ has a bicoloured non-loop edge, then it has a bicoloured loop, or it has no loops at all.
\item
If $\widehat{H}$ has no bicoloured edge, then all unicoloured loops are of the same colour.
\item
The bicoloured part $D_{\widehat{H}}$ is a bi-arc tree.
\item
The blue part $B_{\widehat{H}}$ is a bi-arc tree.
\item
$\widehat{H}$ contains no signed tree from the family $\cal T$.
\item
The vertices with red (respectively blue, respectively bicoloured) loops induce a connected subgraph of $\widehat{H}$.
\end{enumerate}

If at least one of the conditions fails, then $\textsc{List-S-Hom}(\widehat{H})$ is NP-complete.
\end{corollary}

We now return to the proof of Theorem~\ref{provable}.
 
\begin{proof}
We show that for a good signed tree $\widehat{H}$, the problem $\textsc{List-S-Hom}(\widehat{H})$ is polynomial-time solvable.
We may assume there is a bicoloured loop, else the result follows from Theorems~\ref{thm:bordeaux} and~\ref{thm:main_irref}.
By Propositon~\ref{description} we distinguish two cases.

For the first case, let $\widehat{H}$ be a good signed tree obtained from a good reflexive caterpillar (with spine $v_1, v_2, \dots, v_k$) by removing 
loops at a subset $S$ of leaves, and optionally replacing any bicoloured edges $v_iu$ by blue edges for the leaves $u \in S$. As in 
the case of reflexive trees, we use a special min ordering of $\widehat{H}$. This means that if a vertex $v_i$ (with $1 \leq i \leq d-1$) 
has a non-loop neighbour $u$ connected by unicoloured edge, then $u$ is ordered to come after $v_{i+1}$ in the special min ordering. 
Now we can use our algorithm for reflexive trees, with the observation that if there is a negative cycle $C$ mapped to a unicoloured 
edge $v_iu$, then we can remove $u$ from lists of all vertices in $C$ and continue in modifying the images of such cycles.

For the second case, let $\widehat{H}$ be a good signed $2$-caterpillar (having spine $v_1, v_2, \dots, v_k$), obtained from a bi-arc tree by
replacing edges and loops by bicoloured edges and loops (respectively), optionally adding unicoloured loops at 
leaves of $v_1$, and then adding blue edges from the spine or children of the spine to loopless leaves. We set 
$T = V(\widehat{H})$ and $T' = V(D_{\widehat{H}})$; moreover, we set $L = T \setminus T'$. It follows from 
Corollary~\ref{mainko} that all vertices of $L$ are loopless leaves in $T$ incident with exactly one blue edge. 
We also note that if two distinct vertices $a$ and $b$ in $T'$ are adjacent in $\widehat{H}$, then they are 
adjacent by a bicoloured edge.

To prove $\textsc{List-S-Hom}(\widehat{H})$ is polynomial-time solvable, 
we shall construct a suitable majority polymorphism. Recall that for a signed graph 
$\widehat{H}$, the switching graph $S(\widehat{H})$ is constructed as follows. We represent $\widehat{H}$ as 
$(H, \pi)$ where the signature $\pi$ has all unicoloured non-loop edges blue (positive), and define $S(\widehat{H})$ 
to be the edge-coloured graph $(H^+,\pi^+)$ in which each vertex $x$ of $H$ gives rise to two vertices 
$x, x'$ of $H^+$ and each edge $xy$ of $H$ gives rise to edges $xy, x'y'$ of the colour $\pi(xy)$ in $H^+$
and edges $xy', x'y$ of the opposite colour; this definition also applies to loops, by letting $x=y$. For any 
vertex $x$ of $H$, we shall denote by $x^*$ one of $x, x'$, and by $s(x^*)$ the other one of $x, x'$.
A majority polymorphism of $(H^+,\pi^+)$ is a ternary mapping $F$ on the vertices of $H^+$ such that 
$F(x^*,y^*,z^*)$ is adjacent to $F(u^*,v^*,w^*)$ in blue (red) provided $x^*$ is adjacent to $u^*$ 
in blue (red), $y^*$ is adjacent to $v^*$ in blue (red), and $z^*$ is adjacent to $w^*$ in blue (red, 
respectively), and such that if two arguments from $x^*, y^*, z^*$ are equal, then the assigned value 
$F(x^*,y^*,z^*)$ is also equal to it. A semi-conservative majority polymorphism assigns $F(x^*,y^*,z^*)$ 
to be one of the values $x^*,y^*,z^*, s(x^*), s(y^*), s(z^*)$, and a conservative majority polymorphism 
assigns $F(x^*,y^*,z^*)$ to be one of the values $x^*,y^*,z^*$. As outlined in Section~\ref{back},
if the edge-coloured graph $(H^+,\pi^+)$ admits a semi-conservative majority polymorphism, then the 
signed list homomorphism problem for $(H,\pi)$ is polynomial-time solvable. We shall in fact construct 
a conservative majority polymorphism of $(H^+,\pi^+)$.

To construct a conservative majority polymorphism $F(x^*,y^*,z^*)$ for triples $(x^*,y^*,z^*)$ from $V(H^+)$, 
we will of course define values of triples with repetition to be the repeated value,
$$F(x^*,y^*,y^*)=F(y^*,x^*,y^*)=F(y^*,y^*,x^*)=y^*.$$

Now we partition the triples $(x^*,y^*,z^*)$ of distinct vertices of $V(H^+)$ into two sets $R_1$ and $R_2$, where 
$R_1$ consists of those triples $(x^*,y^*,z^*)$ for which at most one of $x, y, z$ is in $L$, and $R_2$ consists 
of triples that have at least two of $x, y, z$ in $L$. (The vertices $x, y, z$ of $\widehat H$ need not be distinct, as 
long as $x^*, y^*, z^*$ are distinct.) Note that two triples $(x_1^*,y_1^*,z_1^*), (x_2^*,y_2^*,z_2^*)$ that are
coordinate-wise adjacent in $\widehat H$ cannot both be in $R_2$, and if they are both in $R_1$, then there
is a coordinate $t \in \{x, y, z\}$ such that $t_1=t_2$, or the edge $t_1t_2$ is bicoloured in $\widehat H$.

The definition of $F(x^*,y^*,z^*)$ will differ for triples with $(x,y,z) \in R_1$, where we explicitly describe the 
value $F(x^*,y^*,z^*)$, and for triples with $(x,y,z) \in R_2$, where we merely prove that a suitable value 
$F(x^*,y^*,z^*)$ exists.

First we consider the underlying unsigned tree of $\widehat{H}$. It clearly contains all edges of  $B_{\widehat{H}}$,
but it also contains loops that are red in $\widehat{H}$. By Corollary~\ref{monitor}, this tree (with vertex 
set $T$), which we also denote by $T$, is a bi-arc tree, and hence has a majority polymorphism $f$~\cite{feder2003bi,feder2007bitr}. 

We now describe the polymorphism $f$ from~\cite{feder2003bi}, assuming, as above, that the bi-arc tree $T$ is one
of the trees in Figure~\ref{fig:biarc_fig6}.

In both cases, $T$ is a $2$-caterpillar with spine $v_1, v_2, \dots, v_k$, 
on which only $v_1$ has a loop, and either there is only one additional loop on a child of $v_1$ (which may have 
children), or any number of loops on children of $v_1$ which must be leaves. We denote by $T_i$ the subtree rooted
at the vertex $v_i$ of the spine as shown in Figure~\ref{fig:mixed_from_irref}.  We note that in this notation we include
the spine vertex $v_i$ in the tree $T_i$, unlike the convention in~\cite{feder2003bi}, where the spine vertex is explicitly
excluded. (Compare Figure~\ref{fig:biarc_fig6} and~\ref{fig:mixed_from_irref}.) To be able to conveniently apply the results of~\cite{feder2003bi}, we refer to 
the vertices in $T_i$ other than the root $v_i$ as being {\em inside} $T_i$. We also remark that in Section~\ref{sec:irreflexive} we used 
yet another convention for naming subtrees $T_i$ --- they were rooted subtrees at individual children of spine vertices,
cf.\ Figure~\ref{fig:example_good_irref}. In any event, in the current context each subtree $T_i$ is ordered by depth first search (in the case of $T_1$ 
giving higher priority to vertices with loops), and a total ordering of $T$ is obtained by concatenating these DFS orderings 
from $T_1$ to $T_2$ and so on. We also colour the vertices of $T$ by two colours, in a proper colouring ignoring the 
self-adjacencies due to the loops. The value $f(x,y,z)$ is defined as the majority of $x, y, z$ if two of the arguments 
$x, y, z$ are equal, and otherwise it is defined according to the following rules.

\vspace{2mm}

\noindent {\bf Rule (A)} Assume $x, y, z$ are distinct and in the same colour class. 
Let $r(x)$, $r(y)$, $r(z)$ be the (not necessarily distinct) roots of the trees containing $x, y, z$ respectively, and let $v_m$
be the median of these vertices on the spine. Then $f(x,y,z)$ is the vertex from amongst $x, y, z$ in the tree $T_m$, 
and if there are more than one in $T_m$, it is the first vertex in the DFS ordering unless one of the following occurs,
in which case it is the second vertex in the DFS ordering.
      
\begin{itemize}
  \item All three vertices $x, y, z$ lie inside $T_m$ with $m \geq 2$;
  \item all three vertices $x, y, z$ lie inside $T_1$ and at most one of 
           them has a loop;
  \item exactly two of $x, y, z$ lie inside $T_1$ and exactly one of them
           has a loop;
  \item exactly two of $x, y, z$ lie inside $T_1$, neither has a loop,
           exactly one of them is adjacent to the unique neighbour 
           of $v_1$ with a loop, and the third vertex of $x, y, z$ is not 
           $v_1$.
\end{itemize}

\noindent {\bf Rule (B)} Assume $x, y, z$ are distinct but not all in the same colour class. 

Then $f(x,y,z)$ is the first vertex in the DFS ordering of the two vertices in the same colour class, 
except when $\{x, y, z\}$ contains $v_1$ and at least one of its leaf neighbours with a loop, in which 
case $f(x,y,z) = v_1$. 

\vspace{2mm}

We now use the above conservative majority $f$ on $T$ to define a conservative majority $F$ on triples
in $R_1$. We say that a vertex $y$ {\em dominates} a vertex $x$ in $\widehat{H}$ if any blue (or red)
neighbour of $x$ is also blue (red respectively) neighbour of $y$.

\vspace{2mm}
\noindent {\bf Rule (1)}  Assume that at least two of $x^*, y^*, z^*$ are equal, say $y^*=z^*$. As mentioned earlier,
we define $F(x^*,y^*,y^*)$ to be the repeated value, 
$$F(x^*,y^*,y^*)=F(y^*,x^*,y^*)=F(y^*,y^*,x^*)=y^*.$$

\vspace{2mm}
\noindent {\bf Rule (2)} Assume that $x^*, y^*, z^*$ are distinct but two of $x, y, z$ are equal. 
Then for triples $(x^*,y^*,z^*)$ we define the value $F$ to be the first version of the repeated vertex, i.e.,
$$F(x^*,s(x^*),y^*)=F(x^*,y^*,s(x^*))= F(y^*,x^*,s(x^*))=x^*,$$
\noindent unless $x \in L$ or $x$ has a unicoloured loop in $\widehat{H}$ and $y$ dominates $x$ in $\widehat{H}$, 
in which case
$$F(x^*,s(x^*),y^*)=F(x^*,y^*,s(x^*))= F(y^*,x^*,s(x^*))=y^*.$$
\noindent (For example $F(x,x',y)=x$ and $F(x',x,y)=x'$, but $F(x,x',y)=F(x',x,y)=y$ if $y$ dominates
$x$ and $x$ has a unicoloured loop or is in $L$.)

\vspace{2mm}
\noindent {\bf Rule (3)} Assume $x^*, y^*, z^*$ are distinct and also $x, y, z$ are distinct.
For triples $(x^*,y^*,z^*)$, we define $F(x^*,y^*,z^*)$ to be the argument in the same 
coordinate as $f(x,y,z)$, except if $f(x,y,z) \in L$ and another vertex $t \in \{x, y, z\}$ dominates $f(x,y,z)$,
in which case  we define $F(x^*,y^*,z^*)$ to be the argument in the same coordinate as $t$. 

It is easy to check that if two triples $(x^*,y^*,z^*)$ and $(u^*,v^*,w^*)$ are coordinate-wise adjacent in 
blue (red) in $S(\widehat{H})$, then $(x,y,z)$ is adjacent to  $(u,v,w)$ in $T$ and hence $f(x,y,z)$ is 
adjacent to  $f(u,v,w)$ in $T$. We now check that  we can also conclude that $F(x^*,y^*,z^*)$ is adjacent 
to $F(u^*,v^*,w^*)$ in blue (red respectively).

{\bf Case 1.} $x, y, z$ are distinct and $u, v, w$ are distinct.

If $f(x,y,z)$ and $f(u,v,w)$ choose the same coordinate, then $F(x^*,y^*,z^*)$ and $F(u^*,v^*,w^*)$ also choose the 
same coordinate, and hence the values are adjacent in the right colour. (This remains true even if one or both of the 
choices $F(x^*,y^*,z^*)$ and $F(u^*,v^*,w^*)$ were modified by domination.) Otherwise, suppose without loss of generality
that $f(x,y,z)=x$ and $f(u,v,w)=v$. Then the vertex $x$ is adjacent in $T$ to both $u$ and $v$ and hence is not a loop-free 
leaf, and similarly $v$ is not a loop-free leaf. If $x \neq v$, this means that the edge $xv$ is bicoloured in $\widehat{H}$ and 
hence $F(x^*,y^*,z^*)=x^*$ is adjacent to $F(u^*,v^*,w^*)=v^*$ in both colours. If $x=v$, the same argument applies if the
loop $xv$ is bicoloured, so let us assume it is unicoloured. In this situation, Proposition \ref{description} implies that 
the vertex $x=v$ must be a leaf child of $v_1$ in $T$, and $u=y=v_1$. This is governed by the special case of Rule 
(B) in the definition of the majority polymorphism $f$, which implies that we would have $f(x,y,z)=v_1$ contradicting 
$f(x,y,z)=x$, so this case does not occur.

{\bf Case 2.} $u, v, w$ are distinct but $x, y, z$ are not distinct, say $x=y$ (but perhaps $x^* \neq y^*$).

This means that $f(u,v,w)$ is adjacent to $x = f(x,y,z)$ in $T$, and $x$ is not a loop-free leaf since it is adjacent to both $u$ 
and $v$. We now observe that if $F(u^*,v^*,w^*)$ was chosen in the same coordinate as $f(u,v,w)$, then $f(u,v,w) \in T'$ and otherwise $F(u^*,v^*, w^*)$ was chosen in the same coordinate as some $t \in T'$ by Rule (3).  Recall $(u,v,w)$ is in $R_1$ so at most one of $u$ and $v$ belongs to $L$ and both are adjacent to $x$. If $f(u,v,w) \in L$, then the other coordinate ($u$ or $v$) is the dominating vertex $t$.  Hence, $f(u,v,w)$ or $t$ is adjacent 
to $x$ by a bicoloured edge in $\widehat{H}$, and $F(u^*,v^*,w^*)$ is adjacent to $x^*$ and to $s(x^*)$ in both colours. 
(Note that $F(x^*,y^*,z^*)$ is $x^*$ regardless of whether $y^*=x^*$ or $y^*=s(x^*)$.)

{\bf Case 3.} Each triple $x, y, z$ and $u, v, w$ has exactly one repetition.

Suppose first that the repetition is in different positions, say $x=y$ and $v=w$. 

Then $f(x,y,z)=x$ is adjacent to $f(u,v,w)=v$ in $T$. If $x \neq v$, then the edge $xv$ is bicoloured in $\widehat{H}$, and 
$F(x^*,y^*,z^*)=x^*$ or $F(x^*,y^*,z^*)=s(x^*)$ and $F(u^*,v^*,w^*)=v^*$ or $F(u^*,v^*,w^*)=s(v^*)$ are 
adjacent in both colours. If $x = v$, then the same argument applies if the loop is bicoloured, and if it is unicoloured, then
Proposition \ref{description} implies that $u=z$ and $u$ has a bicoloured loop and dominates $x$, whence $F(x^*,y^*,z^*)=z^*$ 
and $F(u^*,v^*,w^*)=u^*$ and the adjacency is correct.
On the other hand, if the repetition is in the same positions, say 
$x=y, u=v$, then $F(x^*,y^*,z^*)=x^*$ is adjacent to $F(u^*,v^*,w^*)=u^*$ by the definition of $F$, and hence the edge 
has the correct colour.

{\bf Case 4.} One triple has all vertices the same, say, $x=y=z$ (but possibly $x^* \neq y^*$).

If we also have $u=v=w$, then by the pigeon principle some coordinate 
contains both $F(x^*,y^*,z^*)$ in $(x^*,y^*,z^*)$ and $F(u^*,v^*,w^*)$ in $(u^*,v^*,w^*)$, and so we have the correct 
adjacency.  If $x$ is joined to $u$ by a bicoloured edge in $\widehat{H}$, then $F(x^*,y^*,z^*)$ is joined to $F(u^*,v^*,w^*)$
with the correct adjacency (even in the domination case of Rule 3).  As both $x, u \not\in L$, the only way for the edge joining
them to be unicoloured, is $x=u$ and the edge is a unicoloured loop.  In this case by Proposition~\ref{description}, $w$
dominates $u$ and is joined to $u=x$ with a bicoloured edge in $\widehat{H}$, again ensuring the right adjacency.

\vspace{2mm}

Now we prove that one can extend the definition of $F$ to $R_2$ so that it remains a polymorphism. (It is of course
possible to define each $F(u^*,v^*,w^*)$ for $(u^*,v^*,w^*) \in R_2$ directly, but we found the arguments become more 
transparent if we only verify that a suitable choice for $F(u^*,v^*,w^*)$ is always possible.) 

Consider first values $F(x^*,y^*,z^*)$ with all three vertices $x, y, z$ in $L$. This means that $x$ is incident in $\widehat{H}$ 
with only one (necessarily blue) edge, say $xx_1$, and similarly for blue edges $yy_1, zz_1$. Thus in the switching graph 
$S(\widehat{H})$ the vertex $x$ is incident with only one blue edge, namely $xx_1$, and one red edge, namely $xx'_1$, 
and similarly for $x'$ and for $y, y', z, z'$. Note that $(x_1,y_1,z_1) \in R_1$ because two vertices of $L$ are never adjacent.
To choose the value of $F(x^*,y^*,z^*)$, we only need to take into account the existing values of $F(x_1^{**},y_1^{**},z_1^{**})$,
where $x_1^{**}$ is also either $x_1$ or $x_1'$, and similarly for $y_1^{**}, z_1^{**}$. For example, $(x,y',z')$ is coordinate-wise
adjacent in blue only to $(x_1,y'_1,z'_1)$ and in red only to $(x'_1,y_1,z_1)$, and the choices of $F(x_1,y'_1,z'_1)$
and $F(x'_1,y_1,z_1)$ occur in the same coordinate, by the definition of $F$ on $R_1$; if, say, $F(x_1,y'_1,z'_1)=x_1$
and $F(x'_1,y_1,z_1)=x'_1$, then setting $F(x,y',z')=x$ ensures that $F(x,y',z')=x$ is adjacent to $F(x_1,y'_1,z'_1)=x_1$
in blue and to $F(x'_1,y_1,z_1)=x'_1$ in red, as required. Thus in general we can choose the value $F(x^*,y^*,z^*)$ in 
the same coordinate as $F(x_1^{**},y_1^{**},z_1^{**})$, and satisfy the polymorphism property. (Note that this argument 
applies even if the vertices $x, y, z$ are not distinct.)

It remains to consider the case when exactly two of $x, y, z$ belong to $L$, say $x \in L$ and $y \in L$, with unique (blue) 
neighbours $x_1$ and $y_1$ in $\widehat{H}$, and $z \not\in L$, with neighbours $z_1, \dots, z_p$. We want to show that 
there is a suitable value for each $F(x^*,y^*,z^*)$ that maintains the polymorphism property. In the proofs below, we use the fact 
that $(x^*,y^*,z^*)$ is coordinate-wise adjacent in at least blue to each $(x_1^*,y_1^*,z_i^*)$ and possibly also $(x_1^*,y_1^*,s(z_i^*))$
(if the edge $zz_i$ is bicoloured), and adjacent in at least red to each $(s(x_1^*),s(y_1^*),s(z_i^*))$ and possibly also 
$(s(x_1^*),s(y_1^*),z_i^*)$ (if the edge $zz_i$ is bicoloured). In any event, we again denote the relevant triples by
$(x_1^{**},y_1^{**},z_1^{**})$.

Suppose that $x, y, z$ are of the same colour. We observe that $x$ and $y$ cannot lie on the spine, since they are in $L$.

Consider first the case that $x, y$ lie inside the same tree $T_r$. Recall that we say "inside" to mean $x$ and $y$ are not 
on the spine; thus vertices $x_1$ and $y_1$ also belong to $T_r$, and so $T_r$ is the median tree.
If $z$ also lies inside $T_r$, then $x, y, z$ are all children of $v_r$ or all are grandchildren of $v_r$.  The argument is similar to the case where $F(x^*, y^*, z^*)$ is chosen according to the unique neighbours of $x, y, z$.  In the former case $(x^*, y^*, z^*)$ is adjacent to $(v_r^*, v_r^*, z_i^*)$ and in the latter $x, y, z$ are leaves that do have unique neighbours.
If no neighbour $z_i$ 
of $z$ is in $T_r$, then each value $F(x_1^{**},y_1^{**},z_i^{**})$ is either $x_1^{**}$ or $y_1^{**}$ independently of the 
location of $z_i$, and hence choosing correspondingly $F(x^*,y^*,z^*)=x^*$ or $=y^*$ will ensure the polymorphism property. 
If some neighbour $z_i$ of $z$ lies in $T_r, r > 1$, then $z$ is the root of $T_{r-1},$ or of $T_r$, or of $T_{r+1}$, and in 
this case, we can choose $F(x^*,y^*,z^*)=z^*$. Indeed, in this case, $zz_i$ is bicoloured, and $z_i=x_1=y_1$ (if $z$ is in 
$T_{r-1}$ or $T_{r+1}$) or $zx_1, zy_1$ are also bicoloured (if $z$ is in $T_r$). If $r=1$, then in addition to the previous 
case the vertices $z, z_1, \dots, z_p$ can have loops (the vertices $x, y, x_1, y_1$ do not have loops). Since $x_1$ and 
$y_1$ have the same colour, if the colour of $z_i$ is different (when $z=z_i$), we have the value of 
$F(x_1^{**},y_1^{**},z_i^{**})$ equal to the first or second coordinate, and we can choose $F(x^*,y^*,z^*)$
accordingly. Note that these arguments apply also when $y^*=s(x^*)$, i.e., $x=y$. 

If $x \in T_r, y \in T_{s}$ with $r \neq s$, the arguments are similar. If no neighbour $z_j$ of $z$ lies in $T_r$ or $T_{s}$, 
then we can choose $F(x^*,y^*,z^*)=x^*$ or $=y^*$ or $z^*$, depending on which is the median tree, and if some $z_j$ lies in, 
say, $T_r$, then we can choose $F(x^*,y^*,z^*)=x^*$ or $=z^*$ as above.

Next we consider the case when $x, y, z$ do not have the same colour. We may assume that $x, y$ have different colours,
since if $x, y$ have the same colour and $z$ has a different colour then any  relevant $F(x_1^{**},y_1^{**},z_i^{**})$ is
$x_1^{**}$ or $y_1^{**}$ regardless of $z_i$ and we can choose $F(x^*,y^*,z^*)$ accordingly. However, if $z$ has a loop, 
then we need to consider also $F(x_1^{**},y_1^{**},z^{**})$, which could be $z^{**}$ or $s(z^{**})$ if $z$ is the vertex $v_1$;
in that case we can set $F(x^*,y^*,z^*)=z^*$.

Thus assume without loss of generality that $x, z$ have the same colour, but the colour of $y$ is different. Then unless $z$
has a loop, $F(x_1^{**},y_1^{**},z_i^{**})$ is $x_1^{**}$ or $z_i^{**}$, depending on whether $z_i$ precedes or follows $x_1$
in the DFS ordering. If all neighbours $z_i$ precede $x_1$, or all follow $x_1$, the uniform choice of $F(x_1^{**},y_1^{**},z_i^{**})$
allows one to choose $F(x^*,y^*,z^*)$ accordingly. The only situation when some $z_i$ precedes $x_1$ and another $z_j$ follows 
$x_1$ occurs when $z$ is the root of the tree $T_r$ containing $x$. It is easy to see that in that case we can set $F(x^*,y^*,z^*)=z^*$.
Finally, when $z$ has a loop, then we also need to consider $F(x_1^{**},y_1^{**},z^{**})=z^{**}$ and we set $F(x^*,y^*,z^*)=z^*$.
\qed \end{proof}

\section{Conclusions}

It seems difficult to give a full combinatorial classification of the complexity of list homomorphism problems for general signed graphs. We have accomplished this for signed trees (with possible loops). The polynomial algorithms rely on min ordering or on majority polymorphisms, and neither of the methods alone is sufficient.

\section{Acknowledgements}

We are grateful to two exceptionally helpful referees for their careful reading of our manuscript and their detailed and valuable feedback. 

\bibliography{bibliography}

\end{document}